\documentclass[12pt]{article}
\usepackage{amsfonts}
\usepackage{mathrsfs,dsfont,bbm}
\usepackage[title]{appendix}
\usepackage{authblk}
\usepackage[latin1]{inputenc}
\usepackage{amsmath,amssymb}
\usepackage{dsfont}
\usepackage{latexsym}
\usepackage{enumitem}
\usepackage[active]{srcltx}
\usepackage[
bookmarks=true,         
bookmarksnumbered=true, 
colorlinks=true, pdfstartview=FitV, linkcolor=blue, citecolor=blue,
urlcolor=blue]{hyperref}

\usepackage{graphicx}
\usepackage{subfig}
\newtheorem{theorem}{Theorem}[section]
\newtheorem{corollary}[theorem]{Corollary}
\newtheorem{lemma}[theorem]{Lemma}

\newtheorem{remark}[theorem]{Remark}

\numberwithin{equation}{section}
\def\E{\mathbb{E}}
\def\d{\mathrm{d}}
\def\e{\varepsilon}
\def\det{\mathrm{det}}
\def\D{\mathcal{D}}

\newenvironment{proof}[1][Proof]{\noindent\textit{#1.} }{\hfill \rule{0.5em}{0.5em}}

\begin{document}

	\title{\large\textbf{Asymptotic Properties of the Derivative of Self-Intersection Local Time of Multidimensional Fractional Brownian Motion}}
	
	\author{
		Jiazhen Gu$^1$, Jinchi Jiang$^1$, Qian Yu$^{1,}$\thanks{Corresponding author, Email: qyumath@163.com}\\
		{\textsuperscript{1}\footnotesize\itshape School of Mathematics, Nanjing University of Aeronautics and Astronautics}\\
		{\footnotesize\itshape Nanjing, Jiangsu 211106, P.R. China}
	}
	\date{\vspace{-40pt}}
	\maketitle
	\begin{abstract}
		Let $\{B_t^H,t\geq0\}$ be a $d$-dimensional fractional Brownian motion. We prove that the approximation of the first-order derivative of self-intersection local time, defined as
		$$\alpha_{\e,t}^{(1)}(0)=-\int_0^t\int_0^sp_\e^{(1)}(B_s^H-B_r^H)\d r\d s,$$ 
		where $p_\e^{(1)}(x_1,\cdots,x_d):=\partial _{x_1}p(x_1,\cdots,x_d)$ and $p_\e(x)=(2\pi\e)^{-d/2}e^{|x|^2/2\e},x\in\mathbb{R}^d, d\geq2$ is the heat kernel, exits in $L^2$ sense if and only if $H<\frac{3}{2(1+d)}$ and satisfies three different central limit theorems when normalized by $\e^{\frac d2+1-\frac1H}$ for $H>\frac12$ and $d\geq2$, normalized by $\e^{\frac d2+\frac12-\frac 3{4H}}$ for $\frac{3}{2(1+d)}<H<\frac12$ and $d\geq3$, and normalized by $\log(1/\e)^{-\frac12}$ for the critical case $H=\frac{3}{2(1+d)}$ and $d\geq3$.
	\end{abstract}
	\section{Introduction}
	The fractional Brownian motion (fBm) on $\mathbb{R}^d$ with Hurst parameter $H\in(0,1)$ is a $d$-dimensional centered Caussian process $B^H=\{B_t^H,t\geq0\}$ with independent one-dimensional component processes $B^{H,i},i=1,\cdots,d$ and the covariance function is given by
	$$\E(B_t^{H,i}B_s^{H,i})=\frac12\left(t^{2H}+s^{2H}+|t-s|^{2H}\right).$$
	The self-intersection local time (SLT) of fBm, formally defined by 
	$$\alpha_t(x)=\int_0^t\int_0^s\delta(B_s^H-B_r^H-x)\d r\d s, $$
	where $\delta$ is the Diarac function, was first studied by Rosen \cite{rosen1987intersection} in the planar case. Intuitively, $\alpha_t(x)$ measures the amount of time that the process $B_t^H$ spends intersecting itself on the interval $[0,t]$. It was further investigated by Hu and Nualart \cite{hu2005renormalized}, who has shown that for a $d$-dimensional ($d\geq2$) fBm, $\alpha_t(0)$ exists in $L^2$ whenever the Hurst parameter $H$ satisfies $Hd<1$.  
	
	Motivated by spatial integrals concerning local time, the derivative of self-intersection local time (DSLT) for one-dimensional fBm was introduced by Rosen \cite{rosen2004derivatives}, denoted by 
	$$\alpha_t'(x):=-\int_0^t\int_0^s\delta'(B_t^H-B_s^H-x)\d r\d s.$$
	This version is based on the occupation-time formula. Since the fBm possesses the Tanaka formula, as we all know, the DSLT for fBm takes another form, which was extended by Jung and Markowsky  \cite{jung2014tanaka}.
	$$\alpha_t'(x):=-H\int_0^t\int_0^s\delta'(B_t^H-B_s^H-x)(s-r)^{2H-1}\d r\d s.$$
	
	  Due to its connections with theoretical physics (see \cite{le1986proprietes,le1988fluctuation} and references therein), the study of DSLT for fBm has become a subject of intensive study. Recently, Yu \cite{yu2021higher} introduced the $k$-th-order DLST for fBm
	\begin{align*}
		\alpha_t^{(k)}(x)&=\frac{\partial_k}{\partial_{x_1}^{k_1}\cdots\partial_{x_d}^{k_d}}\int_0^t\int_0^s\delta(B_t^H-B_s^H-x)\d r\d s\\
		&=(-1)^{|k|}\int_0^t\int_0^s\delta^{(k)}(B_t^H-B_s^H-x)\d r\d s,
	\end{align*} 
	where $k=(k_1,\cdots,k_d)$ is a multi-index with all $k_i$ being non-negative integers, $|k|=\sum_{i=1}^dk_i$ and $\delta^{(k)}(x)=\frac{\partial_k}{\partial_{x_1}^{k_1}\cdots\partial_{x_d}^{k_d}}\delta^{(k)}$ is the $k$-th-order partial derivative of $\delta$. To be precise, we define 
	$$p_\e(x)=\frac{1}{(2\pi\e)^\frac d2}e^{\frac{|x|^2}{2\e}}=\frac{1}{(2\pi)^d}\int_{\mathbb{R}^d}e^{\imath\langle x,y\rangle-\e\frac{|y|^2}2}\d y,$$ 
	where $x,y\in\mathrm{R}^d$, $\langle x,y\rangle=\sum_{i=1}^dx_iy_i$, $|y|^2=\sum_{i=1}^dy_i^2$, $\d y=\d y_1\cdots\d y_d$ and $\imath$ is the imaginary unit. Since $p_\e\rightarrow\delta$ as $\e$ tends to zero, we can regard $\alpha_{t,\e}(x)$ as an approximation of $\alpha_{t}$ for a sufficiently minor $\e$. Consider
	\begin{equation}\label{sec1-eq.1}
		\alpha_{\e,t}^{(k)}(x)=(-1)^{|k|}\int_0^t\int_0^sp_\e^{(k)}(B_s^H-B_r^H-x)\d r\d s,
	\end{equation}
	where
	$$p_\e^{(k)}(x)=\frac{\partial_k}{\partial_{x_1}^{k_1}\cdots\partial_{x_d}^{k_d}}p_\e(x)=\frac{\imath^{|k|}}{(2\pi)^d}\int_{\mathbb{R}^d}y_1^{k_1}\cdots y_d^{k_d}e^{\imath\langle x,y\rangle-\e\frac{|y|^2}2}\d y.$$
	Then, the $k$-th-order derivative of $\alpha_t$ can be naturally defined as the limit of $\alpha_{\e,t}^{(k)}(x)$ as $\e$ tends to zero. For the $k$-th-order DSLT under this definition, Yu \cite{yu2021higher} proved the existence conditions of $\alpha_{t,\e}(0)$ in $L^2$ as follows.
    
\begin{theorem}{\cite{yu2021higher}} For $H\in(0,1)$ and $\alpha_{t,\e}^{(k)}(x)$ defined in \eqref{sec1-eq.1}. Let $\#$ denote the number of odd numbers in the multi-index $k=(k_1,\cdots,k_d)$. Suppose that $H<\min\{\frac{2}{2|k|+d},\frac{1}{|k|+d-\#},\frac1d\}$. Then, $\alpha_{t,\e}^{(k)}(0)$ exists in $L^2$.
\end{theorem} 

Note that condition $\frac{2}{2|k|+d}$ is the sharp condition for $d\leq2$. This is because a limit theorem for $\alpha_{t,\e}^{(k)}(x), H=2/3$ was proved by Yu \cite{yu2021higher}, who gave a normalization factor $\big(\log\left(1/\e\right)\big)^{-1}$ when $d=1, |k|=1$ (consistent with the conjecture in \cite{jung2014tanaka}), and when $d=2, |k|=1$, Markowsky
\cite{markow2008} proved the limit theorem for $\alpha_{t,\e}^{(k)}(x), H=1/2$ normalized by the same factor.
\newpage 
In this paper, we focus on the optimal existence condition for $d\geq3$ and some limit theorems for DSLT when $|k|=1$. Without loss of generality, we assume that $k_1=1$ and $k_i=0$ for $i=2,\cdots,d$ for the multi-index $k$. For convenience of writing, we will abbreviate $\alpha_{t,\e}^{(1,0,\cdots,0)}(0)$ as $\alpha_{t,\e}^{(1)}(0)$ in the rest of the paper without causing confusion.  

\textbf{Main results}. The first objective of this paper is to provide a necessary and sufficient condition for the existence of $\alpha_{\e,t}^{(1)}(0)$, which improves the previous result established in \cite{yu2021higher}. Furthermore, we consider some limit theorems of DSLT under divergent conditions.
    
\begin{theorem}\label{sec1-thm.4}
		Suppose that $d\geq3$. Then, the first-order DSLT $\alpha_{t,\e}^{(1)}(0)$ exists in $L^2$ if and only if $H<\frac{3}{2(1+d)}.$
\end{theorem}

\begin{remark}
	When $d=2$, the existence condition here is consistent with that in \cite{yu2021higher}. We will demonstrate that the condition $H<\frac{3}{2(1+d)}, d\geq3$ is sufficient for the existence of $\alpha_{t,\e}^{(1)}(0)$ in $L^2$ sense in Section \ref{sec6} and the necessity can be obtained directly from the limit theorem given in Theorem \ref{sec1-thm.3}.
\end{remark}

The first limit theorem is to analyze the asymptotic behavior of $\alpha_{t,\e}^{(1)}(0)$ as $\e$ tends to zero, when $H>1/2$ and $d\geq2$.
	\begin{theorem}\label{sec1-thm.1}
		Suppose that $H>\frac12$ and $d\geq2$. Then,
		$$
			\e^{\frac d2+1-\frac1H}\alpha^{(1)}_{t,\e}(0)\stackrel{law}\rightarrow\mathcal{N}(0,\sigma^2),
		$$
		as $\e$ tends to zero, where $\sigma^2=\frac{t^{2H}}{2^{d+2}\pi^d}B\left(\frac1H,\frac d2+1-\frac1H\right)^2$ and $B(\cdot,\cdot)$ denotes the Beta function.
	\end{theorem}
	Moreover, we study two limit theorems of the case $\frac{3}{2(d+1)}<H<\frac12$ and the critical case $H=\frac{3}{2(d+1)}$ for $d\geq3$. 
	\begin{theorem}\label{sec1-thm.2}
		Suppose that $\frac{3}{2(1+d)}<H<\frac12$ and $d\geq3$. Then, the random variable
		$$
			\e^{\frac d2+\frac12-\frac{3}{4H}}\alpha^{(1)}_{t,\e}(0)\stackrel{law}\rightarrow\mathcal{N}(0,\hat{\sigma}^2),
		$$
		as $\e$ tends to zero, where $\hat{\sigma}^2$ defined in \eqref{sec4-eq.14} is a constant only depending on $H$, $d$ and $t$.
	\end{theorem}
	\begin{remark}
		From the proof of Theorem \ref{sec1-thm.2}, we see that for every $m\in\mathbb{N}$ and $\frac{3}{2(1+d)}<H<\frac12$, the projection of $m$-th Wiener chaos of $\alpha_{t,\e}$ satisfies a central limit theorem when normalized by $\e^{\frac d2+\frac 12-\frac{3}{4H}}$ (see Corollary \ref{sec4-co.1}). Note that the lower bound of $H$ equals $3/4$ and the multiplicative factor $\e^{\frac d2+\frac 12-\frac{3}{4H}}$ degenerates to $\e^{1-\frac{3}{4H}}$ in the case where $d=1$. This result is consistent with \cite[Theorem 4.2]{jaramillo2017asymptotic}. Moreover, when $H=1/2$, our result is in line with the very recent result presented in \cite[Theorem 1.4]{xu2024central}, provided that we set $|k|=1$.  
	\end{remark}
	\begin{theorem}\label{sec1-thm.3}
		Suppose that $H=\frac{3}{2(1+d)}$ and $d\geq3$. Then, the random variable 
		$$
        \big(\log\left(1/\e\right)\big)^{-\frac12}\alpha^{(1)}_{t,\e}(0)\stackrel{law}\rightarrow\mathcal{N}(0,\bar{\sigma}^2),
		$$
		as $\e$ tends to zero, where $\bar{\sigma}^2$ defined in \eqref{sec5-5.6sigma2} is a constant only depending on $H$, $d$ and $t$.
	\end{theorem}
	\begin{remark}
		As Corollary \ref{sec5-co.1} shows, the chaotic component $I_{2m-1}$ (to be specified later in Lemma \ref{sec2-lm.5}) satisfies $\big(\log\left(1/\e\right)\big)^{-\frac12}I_{2m-1}\stackrel{law}\rightarrow\mathcal{N}(0,\bar{\sigma}_m^2)$ for some constant $\bar{\sigma}_m^2$ defined in Corollary  \ref{sec5-co.1}. It is worth pointing out that $H=3/4$ when $d=1$, and this finding aligns with \cite[Theorem 1.2]{yu2020asymptotic}, which strongly supports the validity of our result.
	\end{remark}

It is surprising to remark that the normalization factor $ (\log\left(1/\e\right))^{-\frac12}$ in Theorem \ref{sec1-thm.3} is different from the normalization factor $ (\log\left(1/\e\right))^{-1}$ in  \cite{markow2008} and \cite{yu2021higher}, because the term that affects the limit behavior in Theorem  \ref{sec1-thm.3} is the whole sequence $\alpha^{(1)}_{t,\e}$, not the first chaos of $\alpha^{(1)}_{t,\e}$. This phenomenon was also revealed by  Jaramillo and Nualart  \cite[P.671]{jaramillo2017asymptotic}.
    
The primary challenge in studying the asymptotic behavior of DSLT for high-dimensional fBm stems from the computational complexity of multiple integrals. To establish Theorem \ref{sec1-thm.4}, we initially show that the sum of chaotic component of $\alpha_{t,\e}^{(1)}(0)$, starting from the 2nd, converges in $L^2$ when $H<\frac{3}{2(1+d)}$, and then analyze the limit of the first Wiener chaos as $\e$ tends to zero. The proof of the central limit theorem for $\e^{\frac d2+1-\frac 1H}\alpha_{t,\e}^{(1)}(0)$ follows from estimations of its $L^2$-norm. To obtain Theorems \ref{sec1-thm.2} and \ref{sec1-thm.3}, we will use the method of chaos expansion, which has been widely applied in the study of convergence and asymptotic properties for both SLT and DSLT (see \cite{das2022existence, hu2001self,jaramillo2019functional}). In addition, we will employ a general central limit theorem established in \cite{hu2005renormalized} (also can be found in \cite[P.125]{nourdin2012normal}) to complete the proof of Theorems \ref{sec1-thm.2} and \ref{sec1-thm.3}.
	
	The properties of $\alpha_{t,\e}^{(k)}(0)$ for $|k|\geq2$ are also of interest. One would expect a sufficient and necessary condition for the existence of $\alpha_{t,\e}^{(k)}(0)$ and some central limit theorems to non-existent situation, but this remains unproven. Nevertheless, we venture the following conjectures.
	\vspace{8pt}\\\textbf{Conjectures}:
	\begin{enumerate}
		\item [(1)] If $|k|\geq1$ and $d\geq3$, $\alpha_{t,\e}^{(k)}(0)-\E\big(\alpha^{(k)}_{t,\e}(0)\big)$ converges in $L^2$ if and only if $H<\frac{3}{2(|k|+d)}$, where we set $\E\big(\alpha^{(k)}_{t,\e}(0)\big)=0$ if $|k|$ is odd.
		\item [(2)] If $\frac{3}{2(|k|+d)}<H<\frac{5+(-1)^{|k|}}{8}$, $|k|\geq1$ and $d\geq3$, $\e^{\frac{d+|k|}{2}-\frac 3{4H}}\big[\alpha^{(k)}_{t,\e}(0)-\E\big(\alpha^{(k)}_{t,\e}(0)\big)\big]$ converges in distribution to a normal law as $\e$ tends to zero.
		\item [(3)] If $H=\frac{3}{2(|k|+d)}$, $|k|\geq1$ and $d\geq3$, $\big(\log\left(1/\e\right)\big)^{-\frac12}\big[\alpha^{(k)}_{t,\e}(0)-\E\big(\alpha^{(k)}_{t,\e}(0)\big)\big]$ converges in distribution to a normal law as $\e$ tends to zero.
	\end{enumerate}
	
	The paper is organized as follows. In section \ref{sec2}, we will recall some preliminary results for the integrals encountered in the proof of the main theorems in three cases, chaotic decomposition of $\alpha_{t,\e}^{(1)}(0)$ and a central limit theorem via chaos expansion. The sections \ref{sec6}, \ref{sec3}, \ref{sec4} and \ref{sec5} are devoted to proving Theorems \ref{sec1-thm.4}, \ref{sec1-thm.1}, \ref{sec1-thm.2} and \ref{sec1-thm.3}, respectively. Finally, some technical lemmas are proved in Section \ref{sec7}. Throughout this paper, the letter $K$ denotes a generic positive finite constant independent of $\e$ and may change from line to line. 
	\section{Preliminary}\label{sec2}
	\textbf{Notations} The following notations are used in the sequel. \begin{enumerate}[itemsep=2pt, parsep=0pt]
		\item [(1)] Throughout the paper, we set
		$$\lambda=\mathrm{Var}(B_s^H-B_r^{H})=|s-r|^{2H},
		\rho=\mathrm{Var}(B_{s'}^H-B_{r'}^{H})=|s'-r'|^{2H},$$
		\begin{align*}
			\mu&=\left|\E\left((B_s^H-B_r^H)(B_{s'}^H-B_{r'}^{H})\right)\right|\\
            &=\frac12\big||s'-r|^{2H}+|s-r'|^{2H}-|s'-s|^{2H}-|r-r'|^{2H}\big|.
		\end{align*}
		\item [(2)] For $v,u_1,u_2>0$, we define
		$$G(v,u_1,u_2)=\left|\E\left(B_{u_1}^{H}(B_{v+u_2}^H-B_v^H)\right)\right|.$$
		\item [(3)] $\D$ denotes the set $\{(r,s,r',s')|0<r<s<t,0<r'<s'<t, r<r'\}.$
	\end{enumerate}
	
	Subsequently, we will present some fundamental lemmas that are essential in establishing Theorem \ref{sec1-thm.4}, \ref{sec1-thm.1}, \ref{sec1-thm.2} and \ref{sec1-thm.3}. The first lemma provides bounds for the quantity $\lambda\rho-\mu^2$, which can be obtained from \cite[Lemma 3.1]{hu2001self} and \cite[Appendix B]{jung2014tanaka}. Here, $\lambda$, $\rho$, and $\mu$ represent the three quantities of the covariance matrix of the increment of fBm.
	\begin{lemma}\label{sec2-lm.1}
		Divide $\D$ into the following three parts
		$$\D=\D_1\cup \D_2\cup \D_3.$$
		\textbf{Case(i)} Suppose that $\D_1=\{(r,r',s,s')\in[0,t]^4|r<r'<s<s'\}$. Let $r'-r=a,s-r'=b,s'-s=c.$ Then, there exists a constant $K_1$ such that
		$$\lambda\rho-\mu^2\geq K_1\left((a+b)^{2H}c^{2H}+a^{2H}(b+c)^{2H}\right)$$
		and
		$$\mu=\frac12\left((a+b+c)^{2H}+b^{2H}-a^{2H}-c^{2H}\right).$$
		\textbf{Case(ii)} Suppose that $\D_2=\{(r,r',s,s')\in[0,t]^4|r<r'<s'<s\}$. Let $r'-r=a,s'-r'=b,s-s'=c.$ Then, there exists a constant $K_2$ such that
		$$\lambda\rho-\mu^2\geq K_2b^{2H}\left(a^{2H}+c^{2H}\right)$$
		and
		$$\mu=\frac12\left((a+b)^{2H}+(b+c)^{2H}-a^{2H}-c^{2H}\right).$$
		\textbf{Case(iii)} Suppose that $\D_3=\{(r,r',s,s')\in[0,t]^4|r<s<r'<s'\}$. Let $s-r=a,r'-s=b,s'-r'=c.$ Then, there exists a constant $K_3$ such that
		$$\lambda\rho-\mu^2\geq K_3(ac)^{2H}$$
		and
		$$\mu=\frac12\left|(a+b+c)^{2H}+b^{2H}-(a+b)^{2H}-(c+b)^{2H}\right|.$$
	\end{lemma}
	\textbf{Wiener chaos expansion} The second lemma is related to the Wiener chaos expansion of $\alpha_{e,\e}^{(1)}(0)$. Before that, we are supposed to explain some notations. We denote by $\mathfrak{H}$ the Hilbert space obtained by taking the completion of the space of step functions endowed with the inner product
	$$\langle\textbf{1}_{[r,s]},\textbf{1}_{[r',s']}\rangle_{\mathfrak{H}}=\E\big((B_s^{H,1}-B_r^{H,1})(B_{s'}^{H,1}-B_{r'}^{H,1})\big),$$
	where $\{B_t^{H,1},t\geq0\}$ is a one-dimensional fBm. The mapping $\textbf{1}_{[0,t]}\mapsto B_t^{H,1}$ can be extended to a linear isometry between $\mathfrak{H}$ and the Gaussian subspace $L^2(\Omega,\mathcal{F},\mathbb{P})$. For any integers $m$, we denote by $\mathfrak{H}^{\otimes m}$ and $\mathfrak{H}^{\odot m}$ the $m$-th tensor product of $\mathfrak{H}$ and the $m$-th symmetric tensor product of $\mathfrak{H}$, respectively. Similarly, for $d$-dimensional fBm $B^H=(B^{H,1},\cdots,B^{H,d})$, we can define the Hilbert space $\mathfrak{H}^d$ and the tensor product spaces $(\mathfrak{H}^d)^{\otimes m}$ and $(\mathfrak{H}^d)^{\odot m}$. For $h=(h^1,\cdots,h^d)\in\mathfrak{H}^d$, we define $B^H(h)=\sum_{j=1}^dB^{H,j}(h^j)$. Then, the operator $h\mapsto B^H(h)$ is a linear isometry between $\mathfrak{H}^d$ and the Gaussian subspace of $L^2(\Omega_1\times\cdots\Omega_d,\mathcal{F}_1\times\cdots\mathcal{F}_d,\mathbb{P}_1\times\cdots\mathbb{P}_d)$. For simplicity and consistency, we will continue to use the notation $L^2(\Omega,\mathcal{F},\mathbb{P})$ to refer to this new Gaussian subspace.

    The $m$-th Wiener chaos of $L^2(\Omega,\mathcal{F},\mathbb{P})$, denoted by $\mathfrak{H}^d_m$, is the closed of subspace $L^2(\Omega,\mathcal{F},\mathbb{P})$ generated by the variables
	$$\left\{\prod_{j=1}^dH_{m_j}(B^{H,j}(h^j))|\sum_{j=1}^dm_j=m,h^j\in\mathfrak{H},\Vert h^j\Vert_{\mathfrak{H}}=1\right\},$$
	where $H_m$ is the $m$-th Hermite polynomial, defined by
	$$H_m(x)=(-1)^me^{\frac{x^2}2}\frac{\d^m}{\d x^m}e^{\frac{x^2}2}.$$
	For $h=(h^1,\cdots,h^d)$, we define the mapping
	$$I_m(h^{\otimes m})=\sum_{i_1,\cdots,i_m=1}^d\prod_{j=1}^dH_{m_j(i_1,\cdots,i_m)}(B^{H,j}(h^j))$$
	for $j=1,\cdots,d$, where $m_j(i_1,\cdots,i_m)$ denotes the number of indices in $(i_1,\cdots,i_m)$ equal to $j$, $\sum_{j=1}^dm_j=m$. The range of $I_m$ is contained in $\mathcal{H}^d_m$. This mapping provides a linear isometry between $(\mathfrak{H}^d)^{\otimes m}$ equipped with $\sqrt{m!}\Vert\cdot\Vert_{(\mathfrak{H}^d)^{\otimes m}}$ and $\mathfrak{H}^d_m$ equipped with $L^2$ norm. Now, the second lemma can be given, which is available in \cite[Lemma 2.2]{yu2024limit}.
	\begin{lemma}\label{sec2-lm.5}
		Given a multi-index $\mathbf{i}_n=(i_1,\cdots,i_n)$, $1\leq i_j\leq d$, we define 
		$$\alpha(\mathbf{i}_{2m-1})=\frac{(2m_1)!\cdots(2m_d)!}{(m_1)!\cdots(m_d)!2^m}$$
		if $n=2m-1$ is odd with the number of components of $\mathbf{i}_{2m-1}$ equal to $1$ is $2m_1-1$ and the number of components equal to $k$ is $2m_k$, for $k=2,\cdots,d$; and $\alpha(\mathbf{i}_{2m-1})=0$ otherwise. Then, we have
		$$
			\alpha_{t,\e}^{(1)}(0)=\sum_{m=1}^\infty I_{2m-1}(f_{2m-1,\e}),
		$$
		where $f_{2m-1,\e}$ is the element of $(\mathfrak{H}^d)^{\otimes (2m-1)}$ given by
		\begin{align*}
			f_{2m-1,\e}(\mathbf{i}_{2m-1};&u_1,\cdots,u_{2m-1})\\
			&=\frac{(2\pi)^{-\frac d2}\alpha(\mathbf{i}_{2m-1})}{(2m-1)!}\int_0^t\int_0^s(\e+|r-s|^{2H})^{-\frac d2-m}\prod_{j=1}^{2m-1}\mathbf{1}_{[r,s]}(u_j)\d r\d s.
		\end{align*}
	\end{lemma}
	
	We then compute the $L^2$-norm of the $(2m-1)$-th Wiener chaos of $\alpha_{\e,t}^{(1)}$ and have the following lemma.
	\begin{lemma}\label{sec2-lm.3}
		The $L^2$-norm of the $(2m-1)$-th Wiener chaos can be interpreted as
		\begin{equation}\label{sec2-eq.1}
			\E\big(I_{2m-1}(f_{2m-1,\e})\big)^2=\frac{4m}{(2\pi)^dd}\left(\begin{array}{c}m+\frac d2-1\\m\end{array}\right)\int_\D\frac{\mu^{2m-1}}{[(\e+\lambda)(\e+\rho)]^{\frac d2+m}}\d r \d s\d r'\d s',
		\end{equation}
		where $\left(\begin{array}{c}m+d/2-1\\m\end{array}\right)=\frac{(m+d/2-1)\times\cdots \times (d/2)}{m!}$ is the generalized binomial coefficient.
	\end{lemma}
	\begin{proof}
		By a straightforward calculation, we have 
		\begin{align*}
			&\E\big(I_{2m-1}(f_{2m-1,\e})\big)^2=(2m-1)!\Vert f_{2m-1,\e}\Vert_{(\mathfrak{H}^d)^{\otimes(2m-1)}}\\
			&=2\cdot(2m-1)!\sum_{m_1+\cdots+m_d=m,m_1\geq1}\frac{(2m-1)!}{(2m_1-1)!(2m_2)!\cdots(2m_d)!}\\
			&\qquad\times\frac{(2\pi)^{-d}\alpha^2(\textbf{i}_{2m-1})}{((2m-1)!)^2}\int_\D\frac{\mu^{2m-1}}{[(\e+\lambda)(\e+\rho)]^{\frac d2+m}}\d r \d s\d r'\d s'\\
			&=\frac{(2\pi)^{-d}}{2^{2m-2}}\left(\sum_{m_1+\cdots+m_d=m,m_1\geq1}\frac{m_1(2m_1)!\cdots(2m_d)!}{(m_1!)^2\cdots(m_d!)^2}\right)\int_\D\frac{\mu^{2m-1}}{[(\e+\lambda)(\e+\rho)]^{\frac d2+m}}\d r \d s\d r'\d s'.
		\end{align*}
		Note that
		$$\sum_{m_1+\cdots+m_d=m,m_1\geq1}\frac{m_1(2m_1)!\cdots(2m_d)!}{(m_1!)^2\cdots(m_d!)^2}=\frac{m}{d}\sum_{m_1+\cdots+m_d=m}\prod_{j=1}^d\left(\begin{array}{c}2m_j\\m_j\end{array}\right).$$
		Then, by the generalized binomial formula, we have
		$$\sum_{m=0}^\infty \left(\begin{array}{c}m+\frac d2-1\\m\end{array}\right)4^mx^m=\frac{1}{(1-4x)^{\frac d2}}=\left(\sum_{n=0}^\infty\left(\begin{array}{c}2n\\n\end{array}\right)x^n\right)^d.$$
		Consequently, by comparing the coefficients in front of $x^m$ on both sides of the above equality, we obtain that
		$$\sum_{m_1+\cdots+m_d=m,m_1\geq1}\frac{m_1(2m_1)!\cdots(2m_d)!}{(m_1!)^2\cdots(m_d!)^2}=\frac{m}{d}\left(\begin{array}{c}m+\frac d2-1\\m\end{array}\right)2^{2m}.$$
		Then, \eqref{sec2-eq.1} follows. The proof is complete.
	\end{proof}
	\vspace{5pt}\\\textbf{Central limit theorem via chaos expansion} We denote by $\mathcal{H}$ the separable Hilbert space. Let $\{e_k,k\geq1\}$ be an orthonormal basis in $\mathcal{H}$ for every $p=0,1,\cdots,m$ and for $f\in\mathcal{H}^{\odot m}$, the contraction of $f$ of order $p$ to be the element of $\mathcal{H}^{\otimes2(n-p)}$ is defined by
	$$f\otimes_pf=\sum_{i_1,\cdots,i_p=1}^\infty\langle f,e_{i_1}\otimes\cdots\otimes e_{i_p}\rangle_{\mathcal{H}^{\otimes p}}\otimes \langle f,e_{i_1}\otimes\cdots\otimes e_{i_p}\rangle_{\mathcal{H}^{\otimes p}}$$
	Obviously, $\mathfrak{H}^{d}$ based on $d$-dimensional fBm is a Hilbert space, whose elements also satisfy the above contraction operation. 
	
	In the seminal paper \cite{nualart2005central}, Nualart and Peccati established a criterion for convergence in distribution to a normal law of a sequence of multiple stochastic integrals. Afterwards, Peccati and Tudor \cite{peccati2005gaussian} gave a multidimensional version of this characterization. Hu and Nualart \cite{hu2005renormalized} extended these results to a slightly more general setting applicable to the renormalization of self-intersection local time of fBm. In the sequel, we will use the central limit theorem established in \cite{hu2005renormalized}, which is the principal method we apply in the proof of Theorems \ref{sec1-thm.2} and \ref{sec1-thm.3}.
	\begin{lemma}\label{sec2-lm.4}
		Let $\{F_k,k\geq1\}$ be a sequence of square integrable and centered random variables with Wiener chaos expansions
		$$F_k=\sum_{m=1}^\infty I_m(f_{m,k}).$$
		Suppose that:
		\begin{enumerate}
			\item[(i)] for every $m\geq1$, $\lim_{k\rightarrow\infty}m!\Vert f_{m,k}\Vert^2_{\mathcal{H}^{\otimes m}}=\sigma_m^2;$
			\vspace{-2mm}
			\item[(ii)] for all $m\geq2$, $p=1,\cdots,m-1$, $\lim_{k\rightarrow\infty}\Vert f_{m,k}\otimes_pf_{m,k}\Vert^2_{\mathcal{H}^{\otimes 2(m-p)}}=0;$
			\item[(iii)] $\sum_{m=1}^\infty \sigma_m^2<\infty;$
			\item[(iv)] $\lim_{N\rightarrow\infty}\lim\sup_{k\rightarrow\infty}\sum_{m=N+1}^\infty m!\Vert f_{m,k}\Vert^2_{\mathcal{H}^{\otimes m}}=0.$
		\end{enumerate}
		Then, $F_k$ converges in distribution to a normal law $\mathcal{N}(0,\sum_{n=1}^\infty \sigma_n^2)$ as $k$ tends to the infinity.
	\end{lemma}
	
	The reader is referred to \cite{nourdin2012normal,nualart2006malliavin} and references therein for more details about the theories of chaos decomposition and central limit theorems via chaos expansion.
	\section{Proof of Theorem \ref{sec1-thm.4}}\label{sec6}
	The proof of Theorem \ref{sec1-thm.4} will be accomplished by the following two steps. Firstly, we will prove in Lemma \ref{sec6-pro.1} that the sum of Wiener chaos, commencing from the second term, converges in $L^2$. Secondly, we will show that the limit of the first chaotic component also exists in $L^2$ sense as $\e$ tends to zero.
	\begin{lemma}\label{sec6-pro.1}
		Suppose that $H<\frac{3}{2(d+1)}$ and $d\geq3$. Then, the random variable
		$$J_\e:=\sum_{m=2}^\infty I_{2m-1}(f_{2m-1,\e})$$
		converges in $L^2$ as $\e$ tends to zero.
	\end{lemma}
	\begin{proof}
		By Lemma \ref{sec2-lm.3} and the preliminary result in \cite[(3.4), (3.5)]{yu2024limit}, we have
		\begin{align*}
			\E\left(J_\e\right)^2&=\E\left(\alpha_{t,\e}^{(1)}(0)\right)^2-\E\big(I_{1}(f_{1,\e})\big)^2\\
			&=\beta_d\int_\D\left(\frac{\mu}{[(\e+\lambda)(\e+\rho)-\mu^2]^{\frac d2+1}}-\frac{\mu}{[(\lambda+\e)(\rho+\e)]^{\frac d2+1}}\right)\d r\d s\d r'\d s',
		\end{align*}
		where the constant $\beta_d=2/(2\pi)^d$. Taking $\e=0$ to maximize the integral over $\D$, we obtain that
		\begin{align*}
			\E\left(J_\e\right)^2&\leq\beta_d\int_{\D}\left(\frac{\mu}{(\lambda\rho-\mu^2)^{\frac d2+1}}-\frac{\mu}{(\lambda\rho)^{\frac d2+1}}\right)\d r\d s\d r'\d s\\
			&\leq K\int_\D\frac{\mu^3}{(\lambda\rho-\mu^2)^{\frac d2+1}\lambda\rho}\d r\d s\d r'\d s'\\
			&=K\sum_{i=1}^3\int_{\D_i}\frac{\mu^3}{(\lambda\rho-\mu^2)^{\frac d2+1}\lambda\rho}\d r\d s\d r'\d s'=:V_1+V_2+V_3.
		\end{align*}
		If we have demonstrated the finiteness of $V_1$, $V_2$ and $V_3$, we would have completed the proof of this lemma. 
		\\\textbf{For the $V_1$ term}. By changing the coordinates $(r,r',s,s')$ by $(r,a=r'-r,b=s-r',c=s'-s)$ and integrating the $r$ variable, we have
		$$V_1\leq K\int_{[0,t]^3}\frac{\mu^3}{(\lambda\rho-\mu^2)^{\frac d2+1}\lambda\rho}\d a\d b\d c.$$
		By the fact that
		$$\mu\leq\sqrt{\lambda\rho}=(a+b)^H(c+d)^H$$
		and Lemma \ref{sec2-lm.1}  Case (i), we have
		\begin{align}\label{sec6-eq.1}
			V_1&\leq K\int_{[0,t]^3}\frac{\mu}{(\lambda\rho-\mu^2)^{\frac d2+1}}\d a\d b\d c\nonumber\\
			&\leq K \int_{[0,t]^3}\frac{(a+b)^H(b+c)^H}{[(a+b)^{2H}c^{2H}+a^{2H}(b+c)^{2H}]^{\frac d2+1}}\d a\d b\d c\nonumber\\
			&\leq K \int_{[0,t]^3}a^{-\frac{Hd}2-H}c^{-\frac{Hd}2-H}((a+b)(b+c))^{-\frac{Hd}2}\d a\d b\d c\nonumber\\
			&\leq K \int_{[0,t]^3}a^{-(1+\alpha)\frac{Hd}2-H}c^{-(1+\alpha)\frac{Hd}2-H}b^{-\beta Hd}\d a\d b\d c,
		\end{align}
		where we use the elementary inequality in the third inequality and the Young inequality $x+y\geq Kx^\alpha y^\beta$ in the last inequality with 
		$$\alpha=\frac{Hd+\delta-1}{Hd},\qquad\beta=\frac{1-\delta}{Hd},$$
		and $0<\delta<3-2Hd-2H$. Since
		$$(1+\alpha)\frac{Hd}2+H=Hd+H-\frac12+\frac{\delta}2<1,\qquad\beta Hd=1-\delta<1,$$
		we obtain that \eqref{sec6-eq.1} is finite.
		\\\textbf{For the $V_2$ term}. By changing the coordinates $(r,r',s,s')$ by $(r,a=r'-r,b=s'-r',c=s-s')$ and integrating the $r$ variable, we have
		$$V_2\leq K\int_{[0,t]^3}\frac{\mu^3}{(\lambda\rho-\mu^2)^{\frac d2+1}\lambda\rho}\d a\d b\d c.$$
		In addition, the fact that $\mu\leq\sqrt{\lambda\rho}$ implies that the finiteness of $V_2$ can be obtained by showing
		$$\int_{[0,t]^3}\frac{\mu}{(\lambda\rho-\mu^2)^{\frac d2+1}}\d a\d b\d c<\infty.$$
		To deal with this, we divide the domain of the integral into two subsets $\mathcal{C}_1$ and $\mathcal{C}_2$, such that $[0,t]^3=\mathcal{C}_1\cup\mathcal{C}_2$, where the sets are defined as 
		$$\mathcal{C}_1=\{(a,b,v)\in[0,t]^3|b\leq a\vee c\},$$
		$$\mathcal{C}_2=\{(a,b,v)\in[0,t]^3|b> a\vee c\}.$$
        If $(a,b,c)\in\mathcal{C}_2$, we will use the following estimation in the subsequent proof,
		\begin{align}\label{sec6-eq.2}
			\mu&=\frac12\left((a+b)^{2H}+(b+c)^{2H}-a^{2H}-c^{2H}\right)\nonumber\\
			&=Hb\int_0^1\left((a+bu)^{2H-1}+(c+bu)^{2H-1}\right)\d u\nonumber\\
			&\leq Kb(a^{2H-1}+c^{2H-1}).
		\end{align}
		where we use the condition $H<1/2$ in the last inequality. By Lemma \ref{sec2-lm.1} Case (ii) and the fact that $Hd+H<\frac32$, we have
		\begin{align*}
			\int_{\mathcal{C}_2}\frac{\mu}{(\lambda\rho-\mu^2)^{\frac d2+1}}\d a\d b\d c&\leq\int_{[0,t]^2}\int_0^{a\vee c}\frac{a^{2H-1}+c^{2H-1}}{b^{Hd+2H-1}(a\vee c)^{Hd+2H}}\d b\d a\d c\\
			&=K \int_{[0,t]^2} \frac{a^{2H-1}+c^{2H-1}}{(a\vee c)^{2Hd+4H-2}}\d a\d c\\
			&\leq K \int_0^t\int_0^a c^{2H-1}a^{-2Hd-4H+2}\d c\d a\\
			&= K\int_0^t a^{-2Hd-2H+2}\d a<\infty.  
		\end{align*}
		For the case where $(a,b,c)\in\mathcal{C}_1$, given the fact that
		\begin{align}\label{sec6-eq.3}
			\mu&=\frac12\left((a+b)^{2H}+(b+c)^{2H}-a^{2H}-c^{2H}\right)\nonumber\\
			&=Hb\int_0^1\left((a+bu)^{2H-1}+(c+bu)^{2H-1}\right)\d u\nonumber\\
			&\leq Kb^{2H},
		\end{align}
		  we have
		\begin{align*}
			\int_{\mathcal{C}_1}\frac{\mu}{(\lambda\rho-\mu^2)^{\frac d2+1}}\d a\d b\d c&\leq K\int_0^t\int_{[0,b]^2}b^{-Hd}(ac)^{-\frac{Hd}2-H}\d a\d c\d b\nonumber\\
			&= K\int_0^t b^{-2Hd-2H+2}\d b<\infty.		
		\end{align*}
		\textbf{For the $V_3$ term}. By changing the coordinates $(r,r',s,s')$ by $(r,a=s-r,b=r'-s,c=s'-r')$, then from Lemma \ref{sec2-lm.1} Case (iii), we have
		\begin{align*}
			V_3&\leq\int_{[0,t]^3}\frac{\mu^3}{(\lambda\rho-\mu^2)^{\frac 
            d2+1}\lambda\rho}\d a\d b\d c\nonumber\\
			&\leq K\int_{[0,t]^3}\mu^3(ac)^{-Hd-4H}\d a\d b\d c.
		\end{align*}
		Then, we split the domain as $[0,t]^3=\mathcal{C}_1\cup\mathcal{C}_2$, where the sets $\mathcal{C}_1$ and $\mathcal{C}_2$ are defined as
		$$\mathcal{C}_1=\{(a,b,c)\in[0,t]^3|b\leq a\vee c\},$$
		$$\mathcal{C}_2=\{(a,b,c)\in[0,t]^3|b> a\vee c\}.$$
		If $(a,b,c)\in\mathcal{C}_1$, we are supposed to show that
		$$\int_{[0,t]^2}\int_0^{a}\mu^3(ac)^{-Hd-4H}\d b\d a\d c+\int_{[0,t]^2}\int_0^{c}\mu^3(ac)^{-Hd-4H}\d b\d a\d c=:I_1+I_2<\infty.$$
		For $I_1$, by the Young inequality, we have
		\begin{align*}
			\mu&=G(a+b,a,c)=\frac12\left|(a+b+c)^{2H}+b^{2H}-(b+c)^{2H}-(a+b)^{2H}\right|\\
            &=H(1-2H)ac\int_0^1\int_0^1(b+au+cv)^{2H-2}\d u\d v\\
			&\leq K a^{2\alpha(H-1)+1}c^{2\beta(H-1)+1},
		\end{align*}
		where 
		$$\alpha=\frac{5-Hd-4H-\delta}{6(1-H)},\qquad\beta=\frac{Hd-2H+1+\delta}{6(1-H)},$$
		with $0<\delta<3-2Hd-2H$. Therefore, we have
		$$
            I_1\leq\int_{[0,t]^2}a^{-Hd-4H-6\alpha(1-H)+4}c^{-Hd-4H-6\beta(1-H)+3}\d a\d c<\infty,
        $$
		because 
		$$Hd+4H+6\alpha(1-H)-4=1-\delta<1,$$
		$$Hd+4H+6\beta(1-H)-3=2Hd+2H-2+\delta<1.$$
		With the same argument, we can prove that $I_2<\infty$. When $(a,b,c)\in\mathcal{C}_2$, by the fact that
		$$\mu=G(a+b,a,c)\leq H(1-2H)acb^{2H-2},$$
		we deduce that
		$$\int_{[0,t]^2}\int_{a\vee c}^t\mu^3(ac)^{-Hd-4H}\d b\d a\d c\leq \int_{[0,t]^2}\int_{a\vee c}^ta^{3-Hd-4H}c^{3-Hd-4H}b^{6H-6}\d a\d b\d c.$$
		Noting that $3-Hd-4H>0$ and that $b=(a\vee b\vee c)$, we find that
		$$
        \int_{[0,t]^2}\int_{a\vee c}^t\mu^3(ac)^{-Hd-4H}\d b\d a\d c\leq\int_{[0,t]^3}(a\vee b\vee c)^{-2Hd-2H}\d a\d b\d c<\infty.
       $$
		The finiteness of the above integral follows from the fact that $2Hd+2H<3$. The proof is complete.
	\end{proof}
	\begin{remark}\label{sec6-rm.1}
		It's worth pointing out that the proof of Lemma \ref{sec6-pro.1} reveals that the integrals in $\D_i$
		$$\int_{\D_i}\frac{\mu}{[(\e+\lambda)(\e+\rho)]^{\frac d2+1}}\d r\d s\d r'\d s'$$
		for $i=1,2$ is finite if we take $\e=0$. Therefore, our problem has been reduced to proving that the integral in $\D_3$ converges as $\e$ tends to zero.
	\end{remark}
	\textbf{Proof of Theorem \ref{sec1-thm.4}} From Remark \ref{sec6-rm.1}, it suffices to demonstrate that
	$$\int_{\D_3}\frac{\mu}{[(\e+\lambda)(\e+\rho)]^{\frac d2+1}}\d r\d s\d r'\d s'$$
	is convergent as $\e$ tends to zero. Denote 
	$$F_{\e}(a,b,c)=\frac{G(a+b,a,c)}{[(\e+a^{2H})(\e+c^{2H})]^{\frac d2+1}}.$$
	Then, By changing the coordinates $(r,r',s,s')$ by $(r,a=s-r,b=r'-s,c=s'-r')$ and Lemma \ref{sec2-lm.1} Case (iii), we simplify our problem by proving that the limit 
	\begin{align*}
		\lim_{\e\rightarrow0}&\int_{0\leq a+b+c\leq t}(t-(a+b+c))F_{\e}(a,b,c)\d a\d b\d c\\
		&=\lim_{\e\rightarrow0}\int_{0\leq a+b+c\leq t}tF_{\e}(a,b,c)\d a\d b\d c\\
		&\qquad-\lim_{\e\rightarrow0}\int_{0\leq a+b+c\leq t}(a+b+c)F_{\e}(a,b,c)\d a\d b\d c=:L_1+L_2
	\end{align*}
	exists.
	We first calculate the limit $L_1$. By the L'Hospital rule, we have
	\begin{align*}
		\lim_{\e\rightarrow0}&t\int_{0\leq a+b+c\leq t}\frac{G(a+b,a,c)}{[(\e+a^{2H})(\e+c^{2H})]^{\frac d2+1}}\d a\d b\d c\\
		&=\lim_{N\rightarrow\infty}t\int_{0\leq a+b+c\leq t}F_{N^{-1}}(a,b,c)\d a\d b\d c\\
		&=\lim_{N\rightarrow\infty}tN^{d+1-\frac{3}{2H}}\int_{0\leq x+y+z\leq tN^{\frac1{2H}}}F_{1}(x,y,z)\d x\d y\d z\\
		&=\lim_{N\rightarrow\infty}\frac{t}{N^{\frac{3}{2H}-d-1}}\int_0^{tN^{\frac1{2H}}}\int_{0\leq x+z\leq u}F_{1}(x,u-x-z,z)\d x\d z\d u\\
		&=\lim_{N\rightarrow\infty}\frac{t^2(3-2Hd-2H)^{-1}}{N^{\frac 1H-d-1}}\int_{0\leq x+z\leq tN^{\frac1{2H}}}F_1(x,tN^{\frac1{2H}}-x-z,z)\d x\d z\\
		&=\lim_{N\rightarrow\infty}\frac{t^4N^{d+1}}{3-2Hd-2H}\int_{0\leq\alpha+\beta\leq1}F_1(tN^{\frac1{2H}}\alpha,tN^{\frac1{2H}}(1-\alpha-\beta),tN^{\frac1{2H}}\beta)\d \alpha\d\beta,
	\end{align*}
	where we make the transformation by replacing $(a,b,c)$ with $(N^{-\frac1{2H}}x,N^{-\frac1{2H}}y,N^{-\frac1{2H}}z)$ in the second equality, change the coordinates $(x,u=x+y+z,z)$ in the third equality, and switch to $(\alpha=tN^{\frac1{2H}}x,\beta=tN^{\frac1{2H}}z)$ in the last equality. Note that
	\begin{align*}
		t^4N^{d+1}&F_1(tN^{\frac1{2H}}\alpha,tN^{\frac1{2H}}(1-\alpha-\beta),tN^{\frac1{2H}}\beta)\\&=\frac{t^{2H+4}N^{d+2}G(1-\beta,\alpha,\beta)}{[(1+t^{2H}N\alpha^{2H})(1+t^{2H}N\beta^{2H})]^{\frac d2+1}}=\frac{t^{4-2Hd-2H}G(1-\beta,\alpha,\beta)}{[(t^{-2H}N^{-1}+\alpha^{2H})(t^{-2H}N^{-1}+\beta^{2H})]^{\frac d2+1}}.
	\end{align*}
	By Lemma \ref{sec7.lm.7-4}, we have that
	$$\frac{G(1-\beta,\alpha,\beta)}{(\alpha\beta)^{Hd+2H}}$$
	is integrable in $\{(\alpha+\beta):0\leq\alpha+\beta\leq1\}$ when $H<\frac{3}{2(1+d)}.$ Therefore, using the dominated convergence theorem yields that
	$$L_1=\frac{t^{4-2Hd-2H}}{3-2Hd-2H}\int_{0\leq\alpha+\beta\leq1}\frac{G(1-\beta,\alpha,\beta)}{(\alpha\beta)^{Hd+2H}}\d \alpha\d \beta<\infty.$$
	In order to verify the existence of $L_2$, by the monotone convergence theorem, we are supposed to show that
	\begin{align*}
		\int_{0\leq a+b+c\leq t}&\frac{G(a+b,a,c)}{a^{Hd+2H-1}c^{Hd+2H}}\d a\d b\d c+\int_{0\leq a+b+c\leq t}\frac{bG(a+b,a,c)}{a^{Hd+2H}c^{Hd+2H}}\d a\d b\d c\\
		&+\int_{0\leq a+b+c\leq t}\frac{G(a+b,a,c)}{a^{Hd+2H}c^{Hd+2H-1}}\d a\d b\d c=:I_1+I_2+I_3\leq\infty.
	\end{align*}
	For $I_1$, we deduce from the relation
	\begin{align*}
		G(x+y,x,z)&=H(1-2H)xz\int_0^1\int_0^1(y+xu+zv)^{2H-2}\d u\d v\\
		&\leq H(1-2H)x^{2\alpha(H-1)+1}z^{2\beta(H-1)+1},
	\end{align*} 
	where 
	$$\alpha=\frac{3-2H-Hd-\delta}{2-2H},\qquad\beta=\frac{Hd-1+\delta}{2-2H},$$
	with $0<\delta<3-2Hd-2H$ that 
	$$I_1\leq\int_{[0,t]^3}a^{1-\delta}c^{2Hd+2H-2+\delta}\d a\d b\d c<\infty.$$
	The finiteness of $I_3$ follows by a similar argument to that for $I_1$. Finally, since
	$$G(x+y,x,z)\leq xzy^{2H-2},$$
	we obtain that
	$$I_2\leq\int_{[0,t]^3}\frac{b^{2H-1}}{a^{Hd+2H-1}c^{Hd+2H-1}}\d a\d b\d c<\infty.$$
	The proof is complete.
	\section{Proof of Theorem \ref{sec1-thm.1}}\label{sec3}
	Before completing the proof of Theorem \ref{sec1-thm.1}, we give some useful lemmas below.
	\begin{lemma}\label{sec3-lm.1}
		Suppose that $H>\frac12$ and $d\geq2$. Then, 
		\begin{equation}\label{sec3-eq.0}
			\lim\limits_{\e\rightarrow0}\E\left(\e^{\frac d2+1-\frac1H}\alpha^{(1)}_{t,\e}(0)\right)^2=\sigma^2,
		\end{equation}
		where $\sigma^2=\frac{t^{2H}}{2^{d+2}\pi^d}B\left(\frac1H,\frac d2+1-\frac1H\right)^2$. 
	\end{lemma}
	\begin{proof}
		Recall the preliminary result in \cite[(3.5)]{yu2024limit}
		$$\E\left(\alpha^{(1)}_{t,\e}(0)\right)^2=V_1(\e)+V_2(\e)+V_3(\e),$$
		where
		$$V_i(\e)=\frac{2}{(2\pi)^d}\int_{\D_i}\det(\e I+\Sigma)^{-\frac d2-1}|\mu|\d r\d s\d r'\d s'.$$
		Then, we split the proof into three parts to consider $V_1(\e)$, $V_2(\e)$ and $V_3(\e)$, respectively.
		\\
		\textbf{For the $V_1(\e)$ term}. By changing the coordinates $(r,r',s,s')$ by $(r,a=r'-r,b=s-r',c=s'-s)$ and integrating the $r$ variable, we have
		$$V_1(\e)\leq K\int_{[0,t]^3}\det(\e I+\Sigma)^{-\frac d2-1}|\mu|\d a\d b\d c.$$
		Note that
		$$\mu\leq \sqrt{\lambda\rho}=(a+b)^H(b+c)^H$$
		and that 
		$$\det(\e I+\Sigma)\geq K\left(\e^2+\e((a+b)^{2H}+(b+c)^{2H})+a^{2H}(c+b)^{2H}+c^{2H}(a+b)^{2H}\right).$$
		Thus, we have
		\begin{align*}
			&\e^{d+2-\frac2H}V_1(\e)\\
            &\leq K\int_{[0,t]^3}\frac{\e^{d+2-\frac2H}(a+b)^{H}(b+c)^{H}}{\left[\e^2+\e((a+b)^{2H}+(b+c)^{2H})+a^{2H}(c+b)^{2H}+c^{2H}(a+b)^{2H}\right]^{\frac d2+1}}\d a\d b\d c\\
			&=K\int_{[0,t\e^{-\frac1{2H}}]^3}\frac{\e^{1-\frac1{2H}}(x+y)^H(y+z)^H}{\left[1+((x+y)^{2H}+(y+z)^{2H})+x^{2H}(y+z)^{2H}+z^{2H}(x+y)^{2H}\right]^{\frac d2+1}}\d x\d y\d z\\
			&\leq K\int_{\mathbb{R}_+^3}\frac{\e^{1-\frac1{2H}}(x+y)^H(y+z)^H}{\left[1+(x+y)^H(y+z)^H(1+(xz)^H)\right]^{\frac d2+1}}\d x\d y\d z,
		\end{align*}
		where we make the change of variables $(a=\e^{\frac1{2H}}x,b=\e^{\frac1{2H}}y,c=\e^{\frac1{2H}}z)$ in the first equality and use the elementary inequality in the last inequality. Since $1-\frac1{2H}>0$ and 
		\begin{equation}\label{sec3-eq.3}
			\int_{\mathbb{R}^3}\frac{(x+y)^H(y+z)^H}{\left[1+(x+y)^H(y+z)^H(1+(xz)^H)\right]^{\frac d2+1}}\d x\d y\d z<\infty
		\end{equation}
		by the fact that $Hd>1$ and $H>\frac12$, we conclude that
		\begin{equation}\label{sec3-eq.2}
			\lim\limits_{\e\rightarrow0}\e^{d+2-\frac2H}V_1(\e)=0.
		\end{equation}
		\textbf{For the $V_2(\e)$ term}. By changing the coordinates $(r,r',s,s')$ by $(r,a=r'-r,b=s'-r',c=s-s')$ and integrating the $r$ variable, we have
		$$V_1(\e)\leq K\int_{[0,t]^3}\det(\e I+\Sigma)^{-\frac d2-1}|\mu|\d a\d b\d c.$$
		By Lemma \ref{sec2-lm.1} Case (ii), we have
		\begin{align*}
			\mu&=\frac12\left((a+b)^{2H}+(b+c)^{2H}-a^{2H}-c^{2H}\right)\\
			&=Hb\int_0^1\left((a+bu)^{2H-1}+(c+bu)^{2H-1}\right)\d u\\
			&\leq Kb(a+b+c)^{2H-1}
		\end{align*}
		and 
		$$\det(\e I+\Sigma)\geq K\left(\e^2+\e((a+b+c)^{2H}+b^{2H})+b^{2H}(a^{2H}+c^{2H})\right).$$
		Then, applying the same technique as in the proof for the $V_1(\e)$ term leads to 
		\begin{align*}
			\e^{d+2-\frac2H}V_2(\e)&\leq K\int_{[0,t]^3}\frac{\e^{d+2-\frac2H}b(a+b+c)^{2H-1}}{\left[\e^2+\e((a+b+c)^{2H}+b^{2H})+b^{2H}(a^{2H}+c^{2H})\right]^{\frac d2+1}}\d a\d b\d c\\
			&=K\int_{[0,t\e^{-\frac1{2H}}]}\frac{\e^{1-\frac1{2H}}y(x+y+z)^{2H-1}}{\left[1+((x+y+z)^{2H}+y^{2H})+y^{2H}(x^{2H}+z^{2H})\right]^{\frac d2+1}}\d x\d y\d z\\
			&\leq K\int_{\mathbb{R}_+^3}\frac{\e^{1-\frac1{2H}}y(x+y+z)^{2H-1}}{\left[1+((x+y+z)^{2H}+y^{2H})+y^{2H}(x^{2H}+z^{2H})\right]^{\frac d2+1}}\d x\d y\d z.
		\end{align*}
		We split the domain of the integral above as $\mathbb{R}_+^3=\mathcal{C}_1\cup\mathcal{C}_2$, where the sets $\mathcal{C}_1$ and $\mathcal{C}_2$ are defined as 
		$$\mathcal{C}_1=\{(x,y,z)\in\mathbb{R}_+^3|y\leq x\vee z\},$$
		$$\mathcal{C}_2=\{(x,y,z)\in\mathbb{R}_+^3|y> x\vee z\}.$$
		When $(x,y,z)\in \mathcal{C}_1$, it holds that $x+y+z\leq 3(x\vee z)$, which, in addition to
		$$1+((x+y+z)^{2H}+y^{2H})+y^{2H}(x^{2H}+z^{2H})\geq 1+(x\vee z)^{2H}(1+y^{2H})$$
		and the fact that $Hd>1$, $H>\frac12$, implies that
		\begin{align*}
			&\int_{\mathcal{C}_1}\frac{y(x+y+z)^{2H-1}}{\left[1+((x+y+z)^{2H}+y^{2H})+y^{2H}(x^{2H}+z^{2H})\right]^{\frac d2+1}}\d x\d y\d z\\
			&\leq K\int_{\mathcal{C}_1}\frac{y(x\vee z)^{2H-1}}{\left[1+(x\vee z)^{2H}(1+y^{2H})\right]^{\frac d2+1}}\d x\d y\d z<\infty.
		\end{align*}
		If $(x,y,z)\in\mathcal{C}_2$, a similar proof can be given. Specifically, the fact that $x+y+z\leq 3y$ together with the inequality
		$$1+((x+y+z)^{2H}+y^{2H})+y^{2H}(x^{2H}+z^{2H})\geq1+y^{2H}(1+(x\vee z)^{2H})$$
		yields the desired result due to the condition $Hd>1, H>\frac12$,
		\begin{align*}
			&\int_{\mathcal{C}_2}\frac{y(x+y+z)^{2H-1}}{\left[1+((x+y+z)^{2H}+y^{2H})+y^{2H}(x^{2H}+z^{2H})\right]^{\frac d2+1}}\d x\d y\d z\\
			&\leq K\int_{\mathcal{C}_2}\frac{y^{2H}}{\left[1+y^{2H}(1+(x\vee z)^{2H})\right]^{\frac d2+1}}\d x\d y\d z<\infty.
		\end{align*}
		Consequently, we obtain
		\begin{equation}\label{sec3-eq.4}
			\lim\limits_{\e\rightarrow0}\e^{d+2-\frac2H}V_2(\e)=0.
		\end{equation}
		\textbf{For the $V_3(\e)$ term}. By changing the coordinates $(r,r',s,s')$ by $(r,a=s-r,b=r'-s,c=s'-r')$, then from Lemma \ref{sec2-lm.1} Case (iii), we have
		\begin{align*}
			\mu&=G(a+b,a,c)=\frac12\left|(a+b+c)^{2H}+b^{2H}-(b+c)^{2H}-(a+b)^{2H}\right|\\
			&=H(2H-1)ac\int_0^1\int_0^1(b+au+cv)^{2H-2}\d u\d v,
		\end{align*}
		which, together with
		$$\det(\e I+\Sigma)=\e^2+\e(a^{2H}+c^{2H})+(ac)^{2H}-G^2(a+b,a,c),$$
		leads to
		\begin{align*}
			V_3(\e)&=\frac{2}{(2\pi)^d}\int_{[0,t]^3}\mathds{1}_{[0,t]}(a+b+c)(t-a-b-c)\det(\e I+\Sigma)^{-\frac d2-1}|\mu|\d a\d b\d c\\
			&=\frac{2}{(2\pi)^d}\int_0^t\int_{[0,t\e^{-\frac1{2H}}]^2}\mathds{1}_{[0,t]}(y+\e^{\frac1{2H}}(x+z))(t-y-\e^{\frac1{2H}}(x+z))\\
			&\qquad\times\e^{\frac1H-d-2}\times\frac{G(\e^{\frac1{2H}}x+y,\e^{\frac1{2H}}x,\e^{\frac1{2H}}z)}{\left[(1+x^{2H})(1+z^{2H})-\e^{-2}G^2(\e^{\frac1{2H}}x+y,\e^{\frac1{2H}}x,\e^{\frac1{2H}}z)\right]^{\frac d2+1}}\d x\d z\d y,
		\end{align*}
		where we change the coordinates $(a,b,c)$ by $(\e^{\frac1{2H}}x,y,\e^{\frac1{2H}}z)$ in the last equality. Denote
		\begin{align*}
			\Phi_\e(x,y,z)&=\mathds{1}_{[0,\e^{-\frac1{2H}}]}(x)\mathds{1}_{[0,\e^{-\frac1{2H}}]}(z)\mathds{1}_{[0,t]}(y+\e^{\frac1{2H}}(x+z))\\
			&\qquad\times\e^{-\frac1H}\times\frac{(t-y-\e^{\frac1{2H}}(x+z))G(\e^{\frac1{2H}}x+y,\e^{\frac1{2H}}x,\e^{\frac1{2H}}z)}{\left[(1+x^{2H})(1+z^{2H})-\e^{-2}G^2(\e^{\frac1{2H}}x+y,\e^{\frac1{2H}}x,\e^{\frac1{2H}}z)\right]^{\frac d2+1}}.
		\end{align*}
		As a result, 
		$$\e^{d+2-\frac2H}V_3(\e)=\frac{2}{(2\pi)^d}\int_0^t\int_{\mathbb{R}_+^2}\Phi_\e(x,y,z)\d x\d z\d y.$$
		Observing that
		\begin{equation}\label{sec3-eq.5}
			G(\e^{\frac1{2H}}x+y,\e^{\frac1{2H}}x,\e^{\frac1{2H}}z)=H(2H-1)\e^{\frac1H}xz\int_0^1\int_0^1(y+\e^{\frac1{2H}}xu+\e^{\frac1{2H}}zv)^{2H-2}\d u\d v
		\end{equation}
		and the fact that $\frac2H-2>0$, we have
		$$
			\lim\limits_{\e\rightarrow0}\Phi_\e(x,y,z)=\frac{H(2H-1)(t-y)y^{2H-2}xz}{\left[(1+x^{2H})(1+z^{2H})\right]^{\frac d2+1}}.
	$$
		By \eqref{sec3-eq.5}, we deduce that there exists a constant $K>0$ only depending on $T$ and $H$ such that
		$$\Phi_\e(x,y,z)\leq K\frac{xzy^{2H-2}}{\left[(1+x^{2H})(1+z^{2H})\right]^{\frac d2+1}}.$$
		The right-hand side of the above inequality is integrable in $[0,t]\times\mathbb{R}_+^2$ due to the condition $Hd>1,H>\frac12$. Consequently, we can obtain the following identity by employing the dominated convergence theorem
		\begin{align}\label{se3-eq.7}
			\lim\limits_{\e\rightarrow0}\e^{d+2-\frac2H}V_3(\e)&=\frac{2}{(2\pi)^d}\int_0^t\int_{\mathbb{R}_+^2}\lim\limits_{\e\rightarrow0}\Phi_\e(x,y,z)\d x\d z\d y\nonumber\\
			&=\frac{2}{(2\pi)^d}\int_0^t\int_{\mathbb{R}_+^2}\frac{H(2H-1)(t-y)y^{2H-2}xz}{\left[(1+x^{2H})(1+z^{2H})\right]^{\frac d2+1}}\d x\d z\d y\nonumber\\
			&=\frac{2H(2H-1)}{(2\pi)^d}\left(\int_0^t(t-y)y^{2H-2}\d y\right)\left(\int_0^{\infty}\frac{x}{(1+x^{2H})^{\frac d2+1}}\right)^2\nonumber\\
			&=\frac{t^{2H}}{2^{d+2}\pi^d}B\left(\frac1H,\frac d2+1-\frac1H\right)^2.
		\end{align}
		Combining \eqref{sec3-eq.2}, \eqref{sec3-eq.4} and \eqref{se3-eq.7}, we complete the proof of \eqref{sec3-eq.0}.
	\end{proof}
	\begin{lemma}\label{sec2-lm.2}
		Suppose that $H>\frac12$ and $d\geq2$. Then, 
		$$
			\lim\limits_{\e\rightarrow0}\E\left(\e^{\frac d2+1-\frac1 H}I_1(f_{1,\e})\right)^2=\sigma^2,
		$$
		where $\sigma^2=\frac{t^{2H}}{2^{d+2}\pi^d}B\left(\frac1H,\frac d2+1-\frac1H\right)^2$.
	\end{lemma}
	\begin{proof}
		Recall the fact that
		$$\E\left(I_1(f_{1,\e})\right)^2=V_1^{(1)}(\e)+V_2^{(1)}(\e)+V_3^{(1)}(\e),$$
		where $V_i^{(1)}(\e)=2\int_{\D_i}\langle f_{1,\e}(\textbf{i}_1;r_1,r_2),f_{1,\e}(\textbf{i}_1,r_1,r_2)\rangle_{\mathfrak{H}^d}\d r_1\d r_2\d s_1\d s_2$ for $i=1,2,3$ and that
		$$0\leq V_i^{(1)}(\e)\leq V_i(\e),$$
		we have by \eqref{sec3-eq.2} and \eqref{sec3-eq.4}
		$$
			\lim\limits_{\e\rightarrow0}\e^{d+2-\frac2 H}\left(V_1^{(1)}(\e)+V_2^{(1)}(\e)\right)=0.
		$$
		Therefore, we only need to consider the term $\e^{d+2-\frac2 H}V_3^{(1)}(\e)$ as $\e\rightarrow0$. By changing the coordinates $(r,r',s,s')$ by $(r,a=s-r,b=r'-s,c=s'-r')$, we have 
		\begin{align*}
			V_3^{(1)}(\e)&=\frac{2}{(2\pi)^d}\int_{[0,t]^3}\frac{\mathds{1}_{[0,t]}(a+b+c)(t-a-b-c)G(a+b,a,c)}{\left[(\e+a^{2H})(\e+c^{2H})\right]^{\frac d2+1}}\d a\d b\d c\\
			&=\frac{2}{(2\pi)^d}\int_0^t\int_{[0,t\e^{\frac1{2H}}]^2} \mathds{1}_{[0,t]}(y+\e^{-\frac1{2H}}(x+z))(t-y-\e^{\frac1{2H}}(x+z))\\
			&\qquad \times\e^{\frac1H-d-2}\times\frac{G(\e^{\frac1{2H}}x+y,\e^{\frac1{2H}}x,\e^{\frac1{2H}}z)}{\left[(1+x^{2H})(1+z^{2H})\right]^{\frac d2+1}}\d x\d z\d y.
		\end{align*}
		Denote
		\begin{align*}
			\Phi_\e(x,y,z)&=\mathds{1}_{[0,\e^{-\frac1{2H}}]}(x)\mathds{1}_{[0,\e^{-\frac1{2H}}]}(z)\mathds{1}_{[0,t]}(y+\e^{\frac1{2H}}(x+z))\\
			&\qquad \times\e^{-\frac1H}\times\frac{(t-y-\e^{\frac1{2H}}(x+z))G(\e^{\frac1{2H}}x+y,\e^{\frac1{2H}}x,\e^{\frac1{2H}}z)}{\left[(1+x^{2H})(1+z^{2H})\right]^{\frac d2+1}}.
		\end{align*}
		As a result,
		$$\e^{d+2-\frac2H}V_3^{(1)}(\e)=\frac{2}{(2\pi)^d}\int_0^t\int_{\mathbb{R}_+^2}\Phi_\e(x,y,z)\d x\d z\d y.$$
		With the same method as in the proof of \eqref{se3-eq.7}, we can draw the conclusion that
		\begin{equation*}
			\lim\limits_{\e\rightarrow0}\e^{d+2-\frac2H}V_3^{(1)}(\e)=\sigma^2.
		\end{equation*}
		We complete the proof of Lemma \ref{sec2-lm.2}.
	\end{proof}
	\vspace{5pt}\\\textbf{Proof of Theorem \ref{sec1-thm.1}}. Recall the fact that
	$$\alpha^{(1)}_{t,\e}(0)=I_1(f_{1,\e})+\sum\limits_{m=2}^\infty I_{2m-1}(f_{2m-1,\e}).$$
	By Lemma \ref{sec2-lm.2}, the variance of $\e^{\frac d2+1-\frac1H}I_1(f_{1,\e})$ converges to $\sigma^2$. Besides, combining Lemma \ref{sec3-lm.1} and \ref{sec2-lm.2}, it holds that
	$$\e^{\frac d2+1-\frac1H}\sum\limits_{m=2}^\infty I_{2m-1}(f_{2m-1,\e})$$
	converges to zero in $L^2$. Theorem \ref{sec1-thm.1} then follows from the fact that $I_1(f_{1,\e})$ is Gaussian. The proof is complete.
	\section{Proof of Theorem \ref{sec1-thm.2}}\label{sec4}
	The core idea of the proof of Theorem \ref{sec1-thm.2} lies in verifying that the family of random variables $\e^{\frac d2+\frac12-\frac3{4H}}\alpha_{t,\e}^{(1)}$ satisfies the conditions of Lemma \ref{sec2-lm.4}. To this end, Lemma \ref{sec4-lm.4} is dedicated to establishing condition (i) and Lemma \ref{sec4-lm.5} is focused on confirming condition (ii). After introducing the lemmas mentioned above, we complete the proof.
	\begin{lemma}\label{sec4-lm.4}
		Suppose that $\frac{3}{2(1+d)}<H<\frac12$ and $d\geq3$. Then
		\begin{equation}\label{sec4-eq.9}
			\lim_{\e\rightarrow0}\E\Big(\e^{\frac d2+\frac12-\frac{3}{4H}}I_{2m-1}(f_{2m-1,\e})\Big)^2=t\int_{\mathbb{R}_+^3}\Psi_m(x,y,z)\d x\d y\d z,
		\end{equation}
		where
		$$\Psi_m(x,y,z)=\sum_{i=1}^3 \frac{4m}{(2\pi)^dd}\left(\begin{array}{c}m+\frac d2-1\\m\end{array}\right)\frac{\mu_i^{2m-1}}{[(1+\lambda_i)(1+\rho_i)]^{\frac d2+m}}$$
		and $\lambda_i$, $\rho_i$ and $\mu_i$ are the parameters $\lambda$, $\rho$ and $\mu$ specific to the domain $\D_i$ with respect to the coordinates $(x=\e^{-\frac1{2H}}a,y=\e^{-\frac1{2H}}b,z=\e^{-\frac1{2H}}c)$, respectively.
	\end{lemma}
	\begin{proof}
		By Lemma \ref{sec2-lm.3}, we have
		\begin{align*}
			&\E\Big(\e^{\frac d2+\frac12-\frac{3}{4H}}I_{2m-1}(f_{2m-1,\e})\Big)^2\\
			&=\frac{4m}{(2\pi)^dd}\left(\begin{array}{c}m+\frac d2-1\\m\end{array}\right)\int_\D\frac{\e^{d+1-\frac{3}{2H}}\mu^{2m-1}}{[(\e+\lambda)(\e+\rho)]^{\frac d2+m}}\d r \d s\d r'\d s'\\
			&=\sum_{i=1}^3 \frac{4m}{(2\pi)^dd}\left(\begin{array}{c}m+\frac d2-1\\m\end{array}\right)\int_{\D_i}\frac{\e^{d+1-\frac{3}{2H}}\mu_i^{2m-1}}{[(\e+\lambda_i)(\e+\rho_i)]^{\frac d2+m}}\d r \d s\d r'\d s'\\
			&=\sum_{i=1}^3 \frac{4m}{(2\pi)^dd}\left(\begin{array}{c}m+\frac d2-1\\m\end{array}\right)\int_{[0,t]^3}\frac{\e^{d+1-\frac{3}{2H}}\mathds{1}_{[0,t]}(a+b+c)(t-a-b-c)\mu_i^{2m-1}}{[(\e+\lambda_i)(\e+\rho_i)]^{\frac d2+m}}\d a \d b\d c,
		\end{align*}
		where we make the change of variables in the same manner as in the proof of Lemma \ref{sec3-lm.1} for the three respective cases in the last equality. Then, changing the coordinates $(a,b,c)$ by $(\e^{\frac1{2H}}x,\e^{\frac1{2H}}y,\e^{\frac1{2H}}z)$ yields that
		\begin{align*}
			\E\Big(\e^{\frac d2+\frac12-\frac{3}{4H}}I_{2m-1}(f_{2m-1,\e})\Big)^2=\int_{0\leq a+b+c\leq t\e^{-\frac1{2H}}}(t-\e^{\frac1{2H}}(x+y+z))\Psi_m(x,y,z)\d x\d y\d z.
		\end{align*}
		By the fact that $\mu<\sqrt{\lambda\rho}$, we have
		$$\Psi_m(x,y,z)\leq\sum_{i=1}^3 \frac{2m}{(2\pi)^dd}\left(\begin{array}{c}m+\frac d2-1\\m\end{array}\right)\frac{\mu_i}{[(1+\lambda_i)(1+\rho_i)]^{\frac d2+1}}.$$
		The right-hand side of the above inequality is integrable in $\mathbb{R}_+^3$ by Lemma \ref{sec4-lm.1}. As a consequence, \eqref{sec4-eq.9} holds due to the dominated convergence theorem.
	\end{proof}
	\begin{lemma}\label{sec4-lm.5}
		Suppose that $\frac{3}{2(1+d)}<H<\frac12$ and $d\geq3$. Then, for $m\geq2$ and $1\leq p\leq2m-2$, we have
		\begin{equation}\label{sec4-eq.10}
			\lim_{\e\rightarrow0}\e^{2d+2-\frac3H}\Vert f_{2m-1,\e}\otimes_p f_{2m-1,\e}\Vert^2_{(\mathfrak{H}^d)^{\otimes2(2m-1-p)}}=0.
		\end{equation}
	\end{lemma}
	\begin{proof}
		Denote the multi-index $(k_1,\cdots,k_p)$ by $\textbf{k}_p$. By the definition of $f_{2m-1,\e}\otimes_p f_{2m-1,\e}$  and Lemma \ref{sec2-lm.5}, we have
		\begin{align*}
			(f_{2m-1,\e}&\otimes_p f_{2m-1,\e})(\textbf{i}_{2m-1-},\textbf{j}_{2m-p};u_1,\cdots,u_{2m-1-p},v_1,\cdots,v_{2m-1-p})\\
			&=\frac{(2\pi)^{-d}}{((2m-1)!)^2}\sum_{k_1,\cdots,k_p=1}^\infty\alpha(\textbf{i}_{2m-1-p},\textbf{k}_p)\alpha(\textbf{j}_{2m-1-p},\textbf{k}_p)\\
			&\qquad\times\int_\D\frac{\mu^p}{[(\e+\lambda)(\e+\rho)]^{\frac d2+m}}\prod_{j=1}^{2m-1-p}\textbf{1}_{[r,s]}(u_j)\textbf{1}_{[r',s']}(v_j)\d r\d s\d r'\d s.
		\end{align*}
		Let the binary set $\mathcal{S}$ be $\mathcal{S}=\{(i,j)|(1,3),(1,4),(2,3),(2,4)\}$. As a result, we can write
		\begin{align}\label{sec4-eq.13}
			\e^{2d+2-\frac3H}\Vert f_{2m-1,\e}\otimes_p f_{2m-1,\e}\Vert^2_{(\mathfrak{H}^d)^{\otimes2(2m-1-p)}}=\e^{2d+2-\frac3H}A^2\E\Bigg(\prod_{(i,j)\in\mathcal{S}}\xi_\e^{i,j}\Bigg),
		\end{align}
		where
		$$A^2=\frac{(2\pi)^{-2d}}{((2m-1)!)^4}\sum_{\textbf{i}_{2m-1-p},\;\textbf{j}_{2m-1-p}}\left(\sum_{\textbf{k}_{p}}\alpha(\textbf{i}_{2m-1-p},\textbf{k}_p)\alpha(\textbf{j}_{2m-1-p},\textbf{k}_p)\right)^2$$
		and the random variables $\xi_\e^{i,j}$ are defined as
		$$\xi_\e^{i,j}=\int_0^t\int_0^s(\e+|r-s|^{2H})^{-\frac d2-m}\prod_{l=1}^p\prod_{k=1}^{2m-1-p}(B_r^{H,i,l}-B_s^{H,i,l})(B_r^{H,j,k}-B_s^{H,j,k})\d r\d s,$$
		where $\{B_t^{H,i,k},t\geq0\}$ is a family of independent one-dimensional fBm with Hurst parameter $H$ for $i=1,2,3,4$ and $k=1,2,\cdots,2m-1$. By Lemma \ref{sec4-lm.4}, we obtain that the variance of the random variables $\e^{\frac d2+\frac12-\frac{3}{4H}}\xi_\e^{i,j}$ converge to a constant as $\e$ tends to zero. In order to prove \eqref{sec4-eq.10}, it suffices to demonstrate that each family of random variables $\e^{\frac d2+\frac12-\frac 3{4H}}\xi_\e^{i,j}$ converges in law to a normal random variable, by \cite[Proposition 1]{peccati2005gaussian}. Observe that the families of random variables $\{\e^{\frac d2+\frac12-\frac 3{4H}}\xi_\e^{i,j},\e>0\}$ have the same distribution as 
		$\{\eta_\e,\e>0\},$ where
		$$\eta_\e=\e^{\frac d2+\frac12-\frac 3{4H}}\int_0^t\int_0^s(\e+|r-s|^{2H})^{-\frac d2-m}\prod_{j=1}^{2m-1}(B_r^{H,j}-B_s^{H,j})\d r\d s.$$
		By the $H$-self-similarity of fBm, we have by changing the coordinates $(r=\e^{\frac1{2H}}x,s=\e^{\frac1{2H}}y)$
		\begin{align*}
			\eta_\e&=\e^{\frac{1}{4H}}\int_0^{t\e^{-\frac1{2H}}}\int_0^y(1+|x-y|^{2H})^{-\frac d2-m}\prod_{j=1}^{2m-1}(B_x^{H,j}-B_y^{H,j})\d x\d y\\
			&=\e^{\frac{1}{4H}}\int_{x<t\e^{-\frac1{2H}}\wedge y}(1+|x-y|^{2H})^{-\frac d2-m}\prod_{j=1}^{2m-1}(B_x^{H,j}-B_y^{H,j})\d x\d y\\
			&\qquad-\e^{\frac{1}{4H}}\int_{x<t\e^{-\frac1{2H}}<y}(1+|x-y|^{2H})^{-\frac d2-m}\prod_{j=1}^{2m-1}(B_x^{H,j}-B_y^{H,j})\d x\d y\\
			&=:\eta_{\e,1}-\eta_{\e,2}.
		\end{align*}
		The condition $Hd+H>\frac32$ yeilds that
		\begin{align*}
			\lim_{\e\rightarrow0}\Vert\eta_{\e,2}\Vert_{L^2}&\leq \lim_{\e\rightarrow0} \e^{\frac{1}{4H}}\int_{x<t\e^{-\frac1{2H}}<y}\frac{|x-y|^{(2m-1)H}}{(1+|x-y|^{2H})^{\frac d2+m}}\d x\d y\\
			&\leq \lim_{\e\rightarrow0} \e^{\frac{1}{4H}}\int_{x<t\e^{-\frac1{2H}}<y}(1+|x-y|^{2H})^{-\frac d2-\frac12}\d x\d y\\
			&= \lim_{\e\rightarrow0} \e^{\frac{1}{4H}}\int_0^{t\e^{-\frac1{2H}}}\int_v^{\infty}(1+u^{2H})^{-\frac d2-\frac12}\d u\d v\\
			&=K\lim_{\e\rightarrow0}\frac{\int_{t\e^{-\frac1{2H}}}^{\infty}(1+u^{2H})^{-\frac d2-\frac12}\d u}{\e^{\frac1{4H}}}\\
			&=K\lim_{\e\rightarrow0}\e^{\frac d2+\frac12-\frac{3}{4H}}=0,
		\end{align*}
		where we make the change of variables $(u=y-x,v=t\e^{-\frac1{2H}}-x)$ in the first equality and use the L'Hospital rule in the second and third equality. Consequently, we have
		$$\eta_{\e,1}=\Lambda_N=\int_0^\infty(1+u^{2H})^{-\frac d2-m}\left(\frac{1}{\sqrt{N}}\int_0^N\prod_{j=1}^{2m-1}(B_{s+u}^{H,j}-B_s^{H,j})\d s\right)\d u,$$
		where we make the change of variables $(u=x-y,s=y)$ and $N=\e^{-\frac{1}{2H}}$. Then, the Lemma \ref{sec4-lm.3} completes the proof.
	\end{proof}
	
	From Lemmas \ref{sec4-lm.4} and \ref{sec4-lm.5}, we have the following corollary.
    
	\begin{corollary}\label{sec4-co.1}
		Suppose that $\frac{3}{2(d+1)}<H<\frac12$ and $d\geq3$. Then, for the $(2m-1)$-th chaotic component $I_{2m-1}(f_{2m-1,\e})$ of $\alpha_{t,\e}^{(1)}(0)$, we have
		\begin{equation}\label{sec4-eq.15}
			\e^{\frac12+\frac d2-\frac{3}{4H}}I_{2m-1}(f_{2m-1,\e})\stackrel{law}\rightarrow\mathcal{N}(0,\hat{\sigma}_m^2),
		\end{equation} 
		where $\hat{\sigma}_m^2=t\int_{\mathbb{R}_+^3}\Psi_m(x,y,z)\d x\d y\d z$ and $\Psi_m(x,y,z)$ is defined in Lemma \ref{sec4-lm.4}.
	\end{corollary}
	\begin{proof}
		By Lemma \ref{sec4-lm.4}, we have 
		$$\lim_{\e\rightarrow0}\E\Big(\e^{\frac d2+\frac12-\frac{3}{4H}}I_{2m-1}(f_{2m-1,\e})\Big)^2=t\int_{\mathbb{R}_+^3}\Psi_m(x,y,z)\d x\d y\d z.$$
		Then, by the \textit{fourth moment theorem} established in \cite{nualart2005central}, we have that \eqref{sec4-eq.15} is equivalent to 
		$$\lim_{\e\rightarrow0}\e^{2d+2-\frac3H}\Vert f_{2m-1,\e}\otimes_p f_{2m-1,\e}\Vert^2_{(\mathfrak{H}^d)^{\otimes2(2m-1-p)}}=0$$
		for every $1\leq p\leq 2m-1$. Therefore, Lemma \ref{sec4-lm.5} completes the proof.
	\end{proof}
	\vspace{5pt}\\\textbf{Proof of Theorem \ref{sec1-thm.2}}. Lemmas \ref{sec4-lm.4} and \ref{sec4-lm.5} imply that conditions (i) and (ii) of Lemma \ref{sec2-lm.4} hold. Thus, we only need to verify conditions (iii) and (iv). By Lemma \ref{sec4-lm.4} and the fact that, 
	$$\E\Big(\e^{\frac d2+\frac12-\frac{3}{4H}}I_{2m-1}(f_{2m-1,\e})\Big)^2\leq t\int_{\mathbb{R}_+^3}\Psi_m(x,y,z)\d x\d y\d z,$$
	it suffices to demonstrate that
	\begin{equation}\label{sec4-eq.2}
		\hat{\sigma}^2:=t\sum_{m=1}^\infty\int_{\mathbb{R}_+^3}\Psi_m(x,y,z)\d x\d y\d z<\infty
	\end{equation}
	to check conditions (iii) and (iv). Denote 
	$$\gamma_i=\frac{\mu_i^2}{(1+\lambda_i)(1+\rho_i)}.$$
	Then, by the fact that $\gamma_i<1$,
	\begin{align}\label{sec4-eq.14}
		\hat{\sigma}^2&=\frac{4t}{(2\pi)^dd}\sum_{i=1}^3\int_{\mathbb{R}_+^3}\left\{\sum_{m=1}^\infty m \left(\begin{array}{c}m+\frac d2-1\\m\end{array}\right)\gamma_i^{m-1}\right\}\frac{\mu_i}{[(1+\lambda_i)(1+\rho_i)]^{\frac d2+1}}\d x\d y\d z\nonumber\\
		&=\frac{4t}{(2\pi)^dd}\sum_{i=1}^3\int_{\mathbb{R}_+^3}\left\{\sum_{m=1}^\infty \frac{\d}{\d\gamma_i} \left(\begin{array}{c}m+\frac d2-1\\m\end{array}\right)\gamma_i^m\right\}\frac{\mu_i}{[(1+\lambda_i)(1+\rho_i)]^{\frac d2+1}}\d x\d y\d z\nonumber\\
		&=\frac{4t}{(2\pi)^dd}\sum_{i=1}^3\int_{\mathbb{R}_+^3}\left(\frac{\d}{\d\gamma_i}\frac{1}{(1-\gamma_i)^{\frac d2}}\right) \frac{\mu_i}{[(1+\lambda_i)(1+\rho_i)]^{\frac d2+1}}\d x\d y\d z\nonumber\\
		&=\frac{2t}{(2\pi)^d}\sum_{i=1}^3\int_{\mathbb{R}_+^3}\frac{\mu_i(1-\gamma_i)^{-\frac d2-1}}{[(1+\lambda_i)(1+\rho_i)]^{\frac d2+1}}\d x\d y\d z\nonumber\\
		&=\frac{2t}{(2\pi)^d}\sum_{i=1}^3\int_{\mathbb{R}_+^3}\frac{\mu_i}{[(1+\lambda_i)(1+\rho_i)-\mu_i^2]^{\frac d2+1}}\d x\d y\d z.
	\end{align}
	Therefore, \eqref{sec4-eq.2} follows from the Lemma \ref{sec4-lm.1}. The proof is complete.
	\section{Proof of Theorem \ref{sec1-thm.3}}\label{sec5}
	The proof of Theorem \ref{sec1-thm.3} closely follows that of Theorem \ref{sec1-thm.2} with a slight modification, which also serves as an application of Lemma \ref{sec2-lm.4}. In fact, Lemmas \ref{sec5-lm.1} and \ref{sec5-lm.2} are devoted to verifying conditions (i) and (ii), respectively. Following this, we will proceed to validate conditions (iii) and (iv), thereby completing the proof of Theorem \ref{sec1-thm.3}.
	\begin{lemma}\label{sec5-lm.1}
		Suppose that $H=\frac{3}{2(1+d)}$ and $d\geq3$. Then, 
		\begin{equation*}
			\lim_{\e\rightarrow0}\frac1{\log\left(1/\e\right)}\E\big(I_{2m-1}(f_{2m-1,\e})\big)^2=t\int_{0\leq\alpha+\beta\leq1}\tilde{\Psi}_m(\alpha,1-\alpha-\beta,\alpha)\d\alpha\d\beta,
		\end{equation*}
		where
		$$\tilde{\Psi}_m(x,y,z)=\sum_{i=1}^3 \frac{2m}{H(2\pi)^dd}\left(\begin{array}{c}m+\frac d2-1\\m\end{array}\right)\frac{\mu_i^{2m-1}}{(\lambda_i\rho_i)^{\frac d2+m}}$$
		and $\lambda_i$, $\rho_i$ and $\mu_i$ are the parameters $\lambda$, $\rho$ and $\mu$ specific to the domain $\D_i$ with respect to the coordinates  $(a=t\e^{\frac1{2H}}\alpha,b=1-\alpha-\beta,c=t\e^{\frac1{2H}}\beta)$, respectively.
	\end{lemma}
	\begin{proof}
		By Lemma \ref{sec2-lm.3}, we have
		\begin{align*}
			&\frac1{\log\left(1/\e\right)}\E\big(I_{2m-1}(f_{2m-1,\e})\big)^2\\
			&=\frac1{\log\left(1/\e\right)}\frac{4m}{(2\pi)^dd}\left(\begin{array}{c}m+\frac d2-1\\m\end{array}\right)\int_\D\frac{\mu^{2m-1}}{[(\e+\lambda)(\e+\rho)]^{\frac d2+m}}\d r \d s\d r'\d s'\\
			&=\frac1{\log\left(1/\e\right)}\sum_{i=1}^3 \frac{4m}{(2\pi)^dd}\left(\begin{array}{c}m+\frac d2-1\\m\end{array}\right)\int_{[0,t]^3}\frac{\mathds{1}_{[0,t]}(a+b+c)(t-a-b-c)\mu_i^{2m-1}}{[(\e+\lambda_i)(\e+\rho_i)]^{\frac d2+m}}\d a \d b\d c,
		\end{align*}
		where we make the change of variables in the same manner as in the proof of Lemma \ref{sec3-lm.1} for the three respective cases in the last equality. Then, changing the coordinates $(a,b,c)$ by $(\e^{\frac1{2H}}x,\e^{\frac1{2H}}y,\e^{\frac1{2H}}z)$ and replacing $1/\e$ with $N$ yields that
		$$\frac1{\log\left(1/\e\right)}\E\big(I_{2m-1}(f_{2m-1,\e})\big)^2=\frac1{\log N}\int_{0\leq x+y+z\leq tN^{\frac1{2H}}}(t-N^{-1}(x+y+z))\Psi_m(x,y,z)\d x\d y\d z,$$
		where $\Psi_m$ is defined in Lemma \ref{sec4-lm.4}. We firstly calculate the limit
		$$\lim_{N\rightarrow\infty}\int_{0\leq x+y+z\leq tN^{\frac1{2H}}}\frac t{\log N}\Psi_m(x,y,z)\d x\d y\d z$$
		and then we will find that the term $N^{-1}(x+y+z)$ gives no contribution to the limit $\lim_{\e\rightarrow0}\log\left(1/\e\right)^{-1}\E\big(I_{2m-1}(f_{2m-1,\e})\big)^2$. By L'Hospital rule, we have
		\begin{align*}
			\lim_{N\rightarrow\infty}\frac t{\log N}&\int_{0\leq x+y+z\leq tN^{\frac1{2H}}}\Psi_m(x,y,z)\d x\d y\d z\\
			&=\lim_{N\rightarrow\infty}\frac t{\log N}\int_0^{tN^{\frac1{2H}}}\int_{0\leq x+z\leq u}\Psi_m(x,u-x-z,z)\d x\d z\d u\\
			&=\lim_{N\rightarrow\infty}\frac {t^2 N^{\frac1{2H}}}{2H}\int_{0\leq x+z\leq tN^{\frac1{2H}}}\Psi_m(x,tN^{\frac1{2H}}-x-z,z)\d x\d z\\
			&=\lim_{N\rightarrow\infty}\frac {t^4 N^{\frac3{2H}}}{2H}\int_{0\leq \alpha+\beta\leq 1}\Psi_m(tN^{\frac1{2H}}\alpha,tN^{\frac1{2H}}(1-\alpha-\beta),tN^{\frac1{2H}}\beta)\d \alpha\d \beta,
		\end{align*}
		where we change the coordinates $(x,u=x+y+z,z)$ in the first equality and switch to $(\alpha=tN^{\frac1{2H}}x,\beta=tN^{\frac1{2H}}z)$ in the last equality. The condition that $Hd+H=\frac32$ yields 
		\begin{align*}
			t^4N^{\frac3{2H}}&\Psi_m(tN^{\frac1{2H}}\alpha,tN^{\frac1{2H}}(1-\alpha-\beta),tN^{\frac1{2H}}\beta)\\
			&=t\sum_{i=1}^3 \frac{4m}{(2\pi)^dd}\left(\begin{array}{c}m+\frac d2-1\\m\end{array}\right)\frac{\mu_i^{2m-1}}{[(t^{-2H}N^{-1}+\lambda_i)(t^{-2H}N^{-1}+\rho_i)]^{\frac d2+m}}.
		\end{align*}
		In addition, by the fact that $\mu<\sqrt{\lambda\rho}$, we have
		$$\tilde{\Psi}_m(\alpha,1-\alpha,\beta)\leq\sum_{i=1}^3 \frac{2m}{H(2\pi)^dd}\left(\begin{array}{c}m+\frac d2-1\\m\end{array}\right)\frac{\mu_i}{(\lambda_i\rho_i)^{\frac d2+1}}.$$
		The right-hand side of the above inequality is integrable in $\{(\alpha,\beta):0\leq\alpha+\beta\leq1\}$ by Lemma \ref{sec7.lm.7-4}, which implies that the term $N^{-1}(x+y+z)$ is irrelevant to the result of limit $\lim_{\e\rightarrow0}\log\left(1/\e\right)^{-1}\E\big(I_{2m-1}(f_{2m-1,\e})\big)^2$. Hence, it holds that
		$$\lim_{N\rightarrow\infty}\frac t{\log N}\int_{0\leq x+y+z\leq tN^{\frac1{2H}}}\Psi_m(x,y,z)\d x\d y\d z=t\int_{0\leq\alpha+\beta\leq1}\tilde{\Psi}_m(\alpha,1-\alpha-\beta,\alpha)\d\alpha\d\beta$$ 
		due to the dominated convergence theorem. The proof is complete. 
	\end{proof}
	\begin{lemma}\label{sec5-lm.2}
		Suppose that $H=\frac{3}{2(1+d)}$ and $d\geq3$. Then, for $m\geq2$ and $1\leq p\leq2m-2$, we have
		\begin{equation*}
			\lim_{\e\rightarrow0}\big(\log(1/\e)\big)^{-2}\Vert f_{2m-1,\e}\otimes_p f_{2m-1,\e}\Vert^2_{(\mathfrak{H}^d)^{\otimes2(2m-1-p)}}=0.
		\end{equation*}
	\end{lemma}
	\begin{proof}
		From \eqref{sec4-eq.13}, we have
		\begin{align*}
			&\big(\log(1/\e)\big)^{-2}\Vert f_{2m-1,\e}\otimes_p f_{2m-1,\e}\Vert^2_{(\mathfrak{H}^d)^{\otimes2(2m-1-p)}}
            =4A^2\big(\log(1/\e)\big)^{-2}\int_{\D^2}\prod_{i=1}^4(\e+{s_i-r_i})^{-\frac d2-m}\\
			&\quad\times\E\left(\prod_{l=1}^p(B_{r_1}^{H,1,l}-B_{s_1}^{H,1,l})(B_{r_2}^{H,1,l}-B_{s_2}^{H,1,l})\right)\E\left(\prod_{l=1}^p(B_{r_3}^{H,2,l}-B_{s_3}^{H,2,l})(B_{r_4}^{H,2,l}-B_{s_4}^{H,2,l})\right)\\
             &\quad\times\E\left(\prod_{k=1}^{2m-p-1}(B_{r_1}^{H,3,k}-B_{s_1}^{H,3,k})(B_{r_3}^{H,3,k}-B_{s_3}^{H,3,k})\right)\\
             &\quad\times\E\left(\prod_{k=1}^{2m-p-1}(B_{r_2}^{H,4,k}-B_{s_2}^{H,4,k})(B_{r_4}^{H,4,k}-B_{s_4}^{H,4,k})\right)	\d r\d s,
		\end{align*}
		where $\d r=\d r_1\d r_2\d r_3\d r_4$, $\d s=\d s_1\d s_2\d s_3\d s_4$, 
        $$\D^2=\{(r,s)\in\mathbb{R}^8|0<r_i<s_i<t,r_1<r_2<r_3<r_4\},$$
        and $\{B_t^{H,i,k},t\geq0\}$ is a family of independent one-dimensional fBm with Hurst parameter $H$ for $i=1,2,3,4$ and $k=1,2,\cdots,2m-2$. Note that the variables satisfy the ordering $r_{i+1}>r_i$ for $i=1,2,3$ within $\D^2$. Therefore, we can rewrite
		\begin{align*}
			\big(\log(1/\e)\big)^{-2}&\Vert f_{2m-1,\e}\otimes_p f_{2m-1,\e}\Vert^2_{(\mathfrak{H}^d)^{\otimes2(2m-1-p)}}=K\big(\log(1/\e)\big)^{-2}\int_{\D^2}\prod_{i=1}^4(\e+|s_i-r_i|)^{-\frac d2-m}\\
			&\times G^p(r_2-r_1,s_1-r_1,s_2-r_2)G^p(r_4-r_3,s_3-r_3,s_4-r_4)\\
			&\times G^{2m-p-1}(r_3-r_1,s_1-r_1,s_3-r_3)G^{2m-p-1}(r_4-r_2,s_2-r_2,s_4-r_4)\d r\d s.
		\end{align*}
		By the Cauchy inequality, we have 
		$$G(v,u_i,u_j)\leq \sqrt{G(v,u_i,u_i)G(v,u_j,u_j)}.$$
		Then, by changing the coordinates $(u_i=s_i-r_i,r_i)$ for $i=1,2,3,4$, we obtain that
		\begin{align*}
			&\big(\log(1/\e)\big)^{-2}\Vert f_{2m-1,\e}\otimes_p f_{2m-1,\e}\Vert^2_{(\mathfrak{H}^d)^{\otimes2(2m-1-p)}}\\
			&=K\big(\log(1/\e)\big)^{-2}\int_{0\leq r_1\leq r_2\leq r_3\leq r_4\leq t}\int_0^{t-r_1}\int_0^{t-r_2}\int_0^{t-r_3}\int_0^{t-r_4}\prod_{i=1}^4(\e+u_i^{2H})^{-\frac d2-m}\nonumber\\
			&\qquad\times G^{p/2}(r_2-r_1,u_1,u_1)G^{p/2}(r_2-r_1,u_2,u_2)G^{p/2}(r_4-r_3,u_3,u_3)G^{p/2}(r_4-r_3,u_4,u_4) \\
            &\qquad\times G^{(2m-p-1)/2}(r_3-r_1,u_1,u_1)G^{(2m-p-1)/2}(r_3-r_1,u_3,u_3)\\
            &\qquad\times G^{(2m-p-1)/2}(r_4-r_2,u_2,u_2)G^{(2m-p-1)/2}(r_4-r_2,u_4,u_4)\d u\d r\\
			&=\frac{K}{N^2(\log N)^2}\int_{0\leq r_1< r_2< r_3< r_4\leq Nt}\\
            &\qquad\quad\int_0^{Nt-r_1}(1+u_1^{2H})^{-\frac d2-m}G^{p/2}(r_2-r_1,u_1,u_1)G^{(2m-p-1)/2}(r_3-r_1,u_1,u_1)\d u_1\\
            &\qquad\times\int_0^{Nt-r_2}(1+u_2^{2H})^{-\frac d2-m}G^{p/2}(r_2-r_1,u_2,u_2)G^{(2m-p-1)/2}(r_4-r_2,u_2,u_2)\d u_2\\
            &\qquad\times\int_0^{Nt-r_3}(1+u_3^{2H})^{-\frac d2-m}G^{p/2}(r_4-r_3,u_3,u_3)G^{(2m-p-1)/2}(r_3-r_1,u_3,u_3)\d u_3\\
            &\qquad\times\int_0^{Nt-r_4}(1+u_4^{2H})^{-\frac d2-m}G^{p/2}(r_4-r_3,u_4,u_4)G^{(2m-p-1)/2}(r_4-r_2,u_4,u_4)\d u_4\d r\\
            &=:\frac{K}{N^2(\log N)^2}\int_{0\leq r_1< r_2< r_3< r_4\leq Nt}\Delta_1\Delta_2\Delta_3\Delta_4\d r,
		\end{align*}
		where we make the change of variables $(\e^{\frac1{2H}}u_i,\e^{\frac1{2H}}r_i)$ for $i=1,2,3,4$ and let $N=\e^{-\frac1{2H}}$ in the second equality. To prove the limit $\lim_{\e\rightarrow0}\big(\log(1/\e)\big)^{-2}\Vert f_{2m-1,\e}\otimes_p f_{2m-1,\e}\Vert^2_{(\mathfrak{H}^d)^{\otimes2(2m-1-p)}}=0$, we first get the upper bound of $\Delta_1\Delta_2\Delta_3\Delta_4$. For $\Delta_1$. It's not hard to see that, when $v>2u$, it holds that
		\begin{align}\label{sec4-eq.8}
			G(v,u,u)&=\frac12\left|(v+u)^{2H}+(v-u)^{2H}-2v^{2H}\right|\nonumber\\
			&=H(1-2H)u^2\int_0^1\int_0^1\left(v+(\eta+\xi-1)u\right)^{2H-2}\d\xi\d\eta\nonumber\\
			&\leq Ku^2v^{2H-2}.
		\end{align} 
        In addition, we have $G(v,u,u)\leq K(u^{2H}+v^{2H})\leq Ku^{2H}$ when $v<2u$. By the fact that $r_2-r_1<r_3-r_1$, we have
        \begin{align*}
            \Delta_1&\leq\left(\int_0^{(r_2-r_1)/2}+\int_{(r_2-r_1)/2}^{(r_3-r_1)/2}+\int_{(r_3-r_1)/2}^{Nt}\right)\\
            &\qquad(1+u_1^{2H})^{-\frac d2-m}G^{p/2}(r_2-r_1,u_1,u_1)G^{(2m-p-1)/2}(r_3-r_1,u_1,u_1)\d u_1=:I_1+I_2+I_3.
        \end{align*}
        By \eqref{sec4-eq.8}, we obtain that
        \begin{align*}
            I_1&\leq\int_0^{(r_2-r_1)/2} (1+u_1^{2H})^{-\frac d2-m}u_1^{2m-1}(r_2-r_1)^{-p(1-H)}(r_3-r_1)^{-(2m-p-1)(1-H)}\d u_1\\
            &\leq\int_0^{(r_2-r_1)/2} u_1^{-\frac32+(2m-1)(1-H)}(r_2-r_1)^{-p(1-H)}(r_3-r_1)^{-(2m-p-1)(1-H)}\d u_1\\
            &=K(r_2-r_1)^{-\frac12+(2m-p-1)(1-H)}(r_3-r_1)^{-(2m-p-1)(1-H)}\\
            &\leq K(r_3-r_1)^{-\frac 12}.
        \end{align*}
        The last inequatliy is due to the fact that $H<\frac12$ and $2m-p-1\geq1$. Similarly, using estimations $G(v,u,u)\leq Ku^2v^{2H-2}$ for $v>2u$ and $G(v,u,u)\leq Ku^{2H}$ for $v<2u$, we have
        \begin{align*}
            I_2&\leq\int_{(r_2-r_1)/2}^{(r_3-r_1)/2} (1+u_1^{2H})^{-\frac d2-m}u_1^{Hp+2m-p-1}(r_3-r_1)^{-(2m-p-1)(1-H)}\d u_1\\
            &\leq\int_{(r_2-r_1)/2}^{(r_3-r_1)/2} u_1^{-\frac32+(2m-p-1)(1-H)}(r_3-r_1)^{-(2m-p-1)(1-H)}\d u_1\\
            &\leq K\left((r_3-r_1)^{-\frac 12}+(r_2-r_1)^{-\frac12+(2m-p-1)(1-H)}(r_3-r_1)^{-(2m-p-1)(1-H)}\right)\\
            &\leq K(r_3-r_1)^{-\frac 12}
        \end{align*}
        and 
        \begin{align*}
            I_3&\leq\int_{(r_3-r_1)/2}^{Nt}(1+u_1^{2H})^{-\frac d2-m}u_1^{H(2m-1)}\d u_1\\
            &\leq \int_{(r_3-r_1)/2}^{Nt}u_1^{-Hd-H}\d u_1\leq K(r_3-r_1)^{-\frac 12}.
        \end{align*}
        Consequently, we have 
        $$\Delta_1\leq K(r_3-r_1)^{-\frac12}.$$
        Applying the same techniques, we can see
        $$\Delta_2\leq K[(r_2-r_1)\vee (r_4-r_2)]^{-\frac12},$$
        $$\Delta_3\leq K[(r_4-r_3)\vee (r_3-r_1)]^{-\frac12}$$
        and
        $$\Delta_4\leq K(r_4-r_2)^{-\frac12}.$$
        Therefore, we obtain that
        \begin{align*}
            \lim_{\e\rightarrow0}&\big(\log(1/\e)\big)^{-2}\Vert f_{2m-1,\e}\otimes_p f_{2m-1,\e}\Vert^2_{(\mathfrak{H}^d)^{\otimes2(2m-1-p)}}\\
            &\leq  \lim_{N\rightarrow\infty}\frac{K}{N^2(\log N)^2}\int_{0\leq r_1< r_2< r_3< r_4\leq Nt}\Delta_1\Delta_2\Delta_3\Delta_4\d r\\
            &\leq  \lim_{N\rightarrow\infty}\frac{K}{N^2(\log N)^2}\int_{0\leq r_1<r_2<r_3<r_3\leq Nt}(r_3-r_1)^{-\frac12}[(r_2-r_1)\vee (r_4-r_2)]^{-\frac12}\\
            &\qquad\qquad\times[(r_4-r_3)\vee (r_3-r_1)]^{-\frac12}(r_4-r_2)^{-\frac12}\d r_1\d r_2\d r_3\d r_4\\
            &\leq \lim_{N\rightarrow\infty}\frac{K}{N^2(\log N)^2}\int_{0\leq r_1<r_2<r_3<r_3\leq Nt}(r_3-r_1)^{-\frac12}(r_2-r_1)^{-\frac12}\\
            &\qquad\qquad\times(r_4-r_3)^{-\frac12}(r_4-r_2)^{-\frac12}\d r_1\d r_2\d r_3\d r_4\\
            &\leq \lim_{N\rightarrow\infty}\frac{K}{N(\log N)^2}\int_{[0,Nt]^3}(x+y)^{-\frac12}x^{-\frac12}z^{-\frac12}(y+z)^{-\frac12}\d x\d y\d z\\
            &\leq \lim_{N\rightarrow\infty}\frac{K}{N(\log N)^2}\left(\int_0^{Nt}x^{-\frac23}\d x\right)^3=0,
        \end{align*}
        where we change the coordinates $(x=r_2-r_1,y=r_3-r_2,z=r_4-r_3,r_1)$ and integrate the $r_1$ variable in the fourth inequality, and use the Young inequality $x+y\geq Kx^{\frac13}y^{\frac23}$ in the last inequality. The proof is complete.
	\end{proof}
    
	\begin{corollary}\label{sec5-co.1}
		Suppose that $H=\frac{3}{2(d+1)}$ and $d\geq3$. Then, for the $(2m-1)$-th chaotic component $I_{2m-1}(f_{2m-1,\e})$ of $\alpha_{t,\e}^{(1)}(0)$, we have
		\begin{equation*}
			\big(\log (1/\e)\big)^{-\frac12}I_{2m-1}(f_{2m-1,\e})\stackrel{law}\rightarrow\mathcal{N}(0,\bar{\sigma}_m^2),
		\end{equation*} 
		where $\bar{\sigma}_m^2=t\int_{0\leq\alpha+\beta\leq1}\tilde{\Psi}_m(\alpha,1-\alpha-\beta,\beta)\d\alpha\d \beta$
		and $\tilde{\Psi}_m(x,y,z)$ is defined in Lemma \ref{sec5-lm.1}.
	\end{corollary}
	The proof of this corollary follows similarly to the proof of Corollary \ref{sec4-co.1}.
	\vspace{5pt}\\\textbf{Proof of Theorem \ref{sec1-thm.3}} Conditions (i) and (ii) have been verified by Lemma \ref{sec5-lm.1} and \ref{sec5-lm.2}, respectively.
	Therefore, we only need to verify conditions (iii) and (iv) of Lemma \ref{sec2-lm.4}. By Lemma \ref{sec5-lm.1} and the fact that
	$$\frac1{\log\left(1/\e\right)}\E\big(I_{2m-1}(f_{2m-1,\e})\big)^2\leq Kt\int_{0\leq\alpha+\beta\leq1}\tilde{\Psi}_m(\alpha,1-\alpha-\beta,\alpha)\d\alpha\d\beta,$$
	it suffices to demonstrate that
	\begin{equation}\label{sec4-eq.12}
		\bar{\sigma}^2:=t\sum_{m=1}^\infty\int_{0\leq\alpha+\beta\leq1}\tilde{\Psi}_m(\alpha,1-\alpha-\beta,\beta)\d\alpha\d \beta<\infty
	\end{equation}
	to check the conditions (iii) and (iv). Denote 
	$$\gamma_i=\frac{\mu_i^2}{(1+\lambda_i)(1+\rho_i)}.$$
	Then, by the fact that $\gamma_i<1$,
	\begin{align}\label{sec5-5.6sigma2}
		\bar{\sigma}^2&=\frac{2t}{H(2\pi)^dd}\sum_{i=1}^3\int_{0\leq\alpha+\beta\leq1}\left\{\sum_{m=1}^\infty m \left(\begin{array}{c}m+\frac d2-1\\m\end{array}\right)\gamma_i^{m-1}\right\}\frac{\mu_i}{(\lambda_i\rho_i)^{\frac d2+1}}\d \alpha\d \beta\nonumber\\
		&=\frac{2t}{H(2\pi)^dd}\sum_{i=1}^3\int_{0\leq\alpha+\beta\leq1}\left\{\sum_{m=1}^\infty \frac{\d}{\d\gamma_i} \left(\begin{array}{c}m+\frac d2-1\\m\end{array}\right)\gamma_i^m\right\}\frac{\mu_i}{(\lambda_i\rho_i)^{\frac d2+1}}\d \alpha\d \beta\nonumber\\
		&=\frac{2t}{H(2\pi)^dd}\sum_{i=1}^3\int_{0\leq\alpha+\beta\leq1}\left(\frac{\d}{\d\gamma_i}\frac{1}{(1-\gamma_i)^{\frac d2}}\right) \frac{\mu_i}{(\lambda_i\rho_i)^{\frac d2+1}}\d \alpha\d \beta\nonumber\\
		&=\frac{t}{H(2\pi)^d}\sum_{i=1}^3\int_{0\leq\alpha+\beta\leq1}\frac{\mu_i}{(\lambda_i\rho_i-\mu_i^2)^{\frac d2+1}}\d \alpha\d \beta.
	\end{align}
	Therefore, \eqref{sec4-eq.12} follows from the Lemma \ref{sec7.lm.7-4}. The proof is complete.

	\section{Technical Lemmas}\label{sec7}
	In this section, we prove some technical lemmas used in the proof of Theorem \ref{sec1-thm.2} and \ref{sec1-thm.3}.
	\begin{lemma}\label{sec4-lm.2}
		Suppose that $\frac{3}{2(1+d)}<H<\frac12$ and $d\geq3$. $\zeta_N(u)$ is defined as
		$$\zeta_N(u):=\frac1{\sqrt{N}}\int_0^N\prod_{j=1}^{2m-1}\left(B_{s+u}^H-B_u^{H}\right)\d s.$$
		Then, $\zeta_N(u)$ converges to a centered Gaussian process $\zeta(u)$ with covariance function
		\begin{equation}\label{sec4-eq.5}
			2\int_0^{\infty}G^{2m-1}(v,u,u)\d v
		\end{equation}
		as $N$ tends to infinity. Moreover, we have
		\begin{equation}\label{sec4-eq.6}
			\int_0^{\infty}(1+u^{2H})^{-\frac d2-m}\sup_N\E(\zeta_N(u))\d u<\infty
		\end{equation}
		and
		\begin{equation}\label{sec4-eq.7}
			\int_0^{\infty}(1+u^{2H})^{-\frac d2-m}\E(\zeta(u))\d u<\infty.
		\end{equation}
	\end{lemma}
	\begin{proof}
		By \cite[Proposition 4]{hu2005renormalized}, we conclude that $\zeta_N(u)$ converges to a centered Gaussian process $\zeta(u)$ with covariance function given by \eqref{sec4-eq.5}. For \eqref{sec4-eq.6}. Note that 
		\begin{align*}
			\E\left(\zeta_N^2(u)\right)&=\frac{2}{N}\int_0^N\int_0^{s_2}G^{2m-1}(s_2-s_2,u,u)\d s_1\d s_2\\
			&=2\int_0^N\left(1-\frac vN\right)G^{2m-1}(v,u,u)\d v,
		\end{align*}
		where we make the change of variables $(v=s_1-s_2,s=s_2)$ and integrate the $s$ variable in the second equality. Then, by \eqref{sec4-eq.8} and using the H\"oder inequality, we obtain that
		\begin{align*}
			\int_0^{\infty}(1+u^{2H})^{-\frac d2-m}&\sup_N\E(\zeta_N(u))\d u\\
			&\leq K\int_0^{\infty}(1+u^{2H})^{-\frac d2-m}\left(\sup_N\E(\zeta_N^2(u))\right)^{1/2}\d u\\
			&\leq K\int_0^{\infty}(1+u^{2H})^{-\frac d2-m}\left(\int_0^{\infty}G^{2m-1}(v,u,u)\d v\right)^{1/2}\d u\\
			&\leq K\int_0^{\infty}(1+u^{2H})^{-\frac d2-m}\left(\int_0^{2u}\left(v^{2(2m-1)H}+u^{2(2m-1)H}\right)\d v\right)^{1/2}\d u\\
			&+ K\int_0^{\infty}(1+u^{2H})^{-\frac d2-m}\left(\int_{2u}^{\infty}\left(v^{(2m-1)(2H-2)}u^{2(2m-1)}\right)\d v\right)^{1/2}\d u\\
			&\leq K\int_0^{\infty}\frac{u^{2(m-1)H+\frac12}}{(1+u^{2H})^{\frac d2+m}}<\infty,
		\end{align*}
		due to the condition $Hd+H>\frac32$. Similarly, we have \eqref{sec4-eq.7}.
	\end{proof}
	\begin{lemma}\label{sec4-lm.3}
		Suppose that $\frac{3}{2(1+d)}<H<\frac12$ and $d\geq3$. Consider the process $\zeta_N(u)$ defined in Lemma \ref{sec4-lm.2}. Then, the random variable
		$$\Lambda_N:=\int_0^{\infty}(1+u^{2H})^{-\frac d2-m}\zeta_N(u)\d u$$
		converges as $N$ tends to infinity in law to 
		$$\Xi:=\int_0^\infty(1+u^{2H})^{-\frac d2-m}\zeta(u)\d u,$$
		which has a normal distribution
		$$\mathcal{N}\left(0,2\int_0^{\infty}\int_0^{\infty}(1+u_1^{2H})^{-\frac d2-m}(1+u_2)^{2H})^{-\frac d2-m}\left(\int_0^{\infty}G^{2m-1}(v,u_1,u_2)\d v\right)\d u_1\d u_2\right).$$
	\end{lemma}
	\begin{proof}
		From the proof of Lemma \ref{sec4-lm.2}, we have
		\begin{align*}
			&\lim_{N\rightarrow\infty}\E\left(\Lambda_N^2\right)=\E\left(\Xi^2\right)\\
			&=2\int_0^{\infty}\int_0^{\infty}(1+u_1^{2H})^{-\frac d2-m}(1+u_2^{2H})^{-\frac d2-m}\left(\int_0^{\infty}G^{2m-1}(v,u_1,u_2)\d v\right)\d u_1\d u_2.
		\end{align*}
		For $M>0$, we define
		$$\Lambda_N^{(M)}:=\int_0^M(1+u^{2H})^{-\frac d2-m}\zeta_N(u)\d u$$
		and
		$$\Xi^{(M)}:=\int_0^M(1+u^{2H})^{-\frac d2-m}\zeta(u)\d u.$$
		It is clear that $\Xi_N^{(M)}$ is a centered Gaussian variable with the variance
		$$2\int_0^M\int_0^M(1+u_1^{2H})^{-\frac d2-m}(1+u_2^{2H})^{-\frac d2-m}\left(\int_0^{\infty}G^{2m-1}(v,u_1,u_2)\d v\right)\d u_1\d u_2.$$
		Let $\phi$ be a continuous and bounded function. By the triangle inequality, we have
		\begin{align*}
			\big|\E\left(\phi(\Lambda_N)-\phi(\Xi)\right)\big|&\leq \big|\E\big(\phi(\Lambda_N)-\phi(\Lambda_N^{(M)})\big)\big|+\big|\E\big(\phi(\Lambda_N^{(M)})-\phi(\Xi_N^{(M)})\big)\big|\\
			&+\big|\E\big(\phi(\Xi^{(M)})-\phi(\Xi)\big)\big|=:\sum_{i=1}^3 Q^{(i)}_{N,M}.
		\end{align*}
		The finiteness of \eqref{sec4-eq.6} and \eqref{sec4-eq.7} shows that $\lim_{M\rightarrow\infty}\sup_N Q^{(i)}_{N,M}=0$, for $i=1,3$, respectively. For $Q^{(2)}_{N,M}$ term.
		Denote
		$$Z_{s,s+u}=\prod_{j=1}^{2m-1}\left(B_{s+u}^H-B_u^{H}\right).$$
		Then, by a straightforward calculation, we obtain that
		\begin{align*}
			\sup_N&\E\left(\zeta_N(u_1)-\zeta_N(u_2)\right)^2\\
			&=\sup_N\frac1N\int_0^N\int_0^N\Big|\E(Z_{s_1,s_1+u_1}-Z_{s_1,s_1+u_2})(Z_{s_2,s_2+u_1}-Z_{s_2,s_2+u_2})\Big|\d s_1\d s_2\\
			&=\sup_N\frac1N\int_0^N\int_0^{s_2}\Big|\E(Z_{s_1,s_1+u_1}Z_{s_2,s_2+u_1})+\E(Z_{s_1,s_1+u_2}Z_{s_2,s_2+u_2})\\
			\displaybreak[0]
			&-\E(Z_{s_1,s_1+u_1}Z_{s_2,s_2+u_2})-\E(Z_{s_1,s_1+u_2}Z_{s_2,s_2+u_1})\Big|\d s_1\d s_2\\
			&=\sup_N\frac2N\int_0^N\int_0^{s_2}\Big|G^{2m-1}(s_2-s_1,u_1,u_1)+G^{2m-1}(s_2-s_1,u_2,u_2)\\
			&-2G^{2m-1}(s_2-s_1,u_1,u_2)\Big|\d s_1\d s_2\\
			&=K\sup_N\int_0^N\left(1-\frac vN\right)\Big|G^{2m-1}(v,u_1,u_1)+G^{2m-1}(v,u_2,u_2)-2G^{2m-1}(v,u_1,u_2)\Big|\d v\\
			&\leq K\int_0^\infty\Big|G^{2m-1}(v,u_1,u_1)+G^{2m-1}(v,u_2,u_2)-2G^{2m-1}(v,u_1,u_2)\Big|\d v.
		\end{align*}
		When $v>2u$, \eqref{sec4-eq.8} shows that
		$G(v,u,u)\leq Ku^{2H}$. It also holds that $G(v,u,u)\leq Ku^{2H}$ if $v<2u$. Combining with the fact that 
		$$G(v,u_1,u_2)\leq\sqrt{G(v,u_1,u_1)G(v,u_2,u_2)}, $$
		we then, by the dominated convergence theorem, conclude that for all $M>0$,
		$$
        \lim_{\delta\to0}\mathop{\sup}_{|u_1-u_2|<\delta\atop u_1,u_2\in[0,M]}\sup_N\E\left(\zeta_N(u_1)-\zeta_N(u_2)\right)^2=0,
        $$ 
		which implies that $\lim_{N\rightarrow\infty} Q_{N,M}^{(2)}=0$. We complete the proof by \cite[Theorem 1.2.1]{billingsley1999convergence}.
	\end{proof}
	\begin{lemma}\label{sec4-lm.1}
		Suppose that $\frac{3}{2(1+d)}<H<\frac12$ and $d\geq3$. Then, for $i=1,2,3$
		\begin{equation}\label{sec4-eq.1}
			\int_{\mathbb{R}_+^3}\frac{\mu_i}{[(1+\lambda_i)(1+\rho_i)-\mu_i^2]^{\frac d2+1}}\d x\d y\d z<\infty,
		\end{equation}
		where $\lambda_i$, $\rho_i$ and $\mu_i$ are the parameters $\lambda$, $\rho$ and $\mu$ specific to the domain $\D_i$ with respect to the coordinates $(x=\e^{-\frac1{2H}}a,y=\e^{-\frac1{2H}}b,z=\e^{-\frac1{2H}}c)$, respectively.
	\end{lemma}
	\begin{proof}
		\textbf{For} {\boldmath{$i=1$}}, by \eqref{sec3-eq.3} we are supposed to show that
		$$I:=\int_{[1,\infty)^3}(x+y)^{-\frac {Hd}2}(y+z)^{-\frac {Hd}2}(xz)^{-\frac{Hd}2-H}<\infty.$$
		Since $d\geq3$ and $Hd+H>\frac32$, we have by the Young inequality
		\begin{align*}
			I&\leq K\int_{[1,\infty)^3}y^{-Hd(\frac23+\frac2{3d})}x^{-\frac{Hd}2(\frac13-\frac2{3d})}z^{-\frac{Hd}2(\frac13-\frac2{3d})}(xz)^{-\frac{Hd}2-H}\d x\d y\d z\\
			&\leq K\int_{[1,\infty)^3}(xyz)^{-\frac23(Hd+H)}<\infty.
		\end{align*}
        
		\textbf{For} {\boldmath{$i=2$}}, similar to the proof of Lemma \ref{sec3-lm.1} $V_2(\e)$ term, we split the domain of integral as $\mathbb{R}_+^3=\mathcal{C}_1\cup\mathcal{C}_2$, where the sets $\mathcal{C}_1$ and $\mathcal{C}_2$ are defined as 
		$$\mathcal{C}_1=\{(x,y,z)\in\mathbb{R}_+^3|y\leq x\vee z\},$$
		$$\mathcal{C}_2=\{(x,y,z)\in\mathbb{R}_+^3|y> x\vee z\}.$$
		When $(x,y,z)\in\mathcal{C}_1$, by Lemma \ref{sec2-lm.1} and the upper bound estimation of $\mu_2$ given by \eqref{sec6-eq.2} 
		$$\mu_2\leq Ky(x^{2H-1}+z^{2H-1}),$$
		we have 
		\begin{align*}
			I_1&:=\int_{[1,\infty)^3}\frac{\mathds{1}_{[1,x\vee z]}(y)\mu_2}{[(1+\lambda_2)(1+\rho_2)-\mu_2^2]^{\frac d2+1}}\d x\d y\d z\\
			&\leq K\int_{[1,\infty)^2}\int_1^{x\vee z}\frac{x^{2H-1}+z^{2H-1}}{y^{Hd+2H-1}(x\vee z)^{Hd+2H}}\d y\d x\d z
		\end{align*}
		If $Hd+2H>2$, we deduce that the above integral is finite since the quadrature function is less than $y^{-Hd-2H+1}(x\vee z)^{-Hd-2H}$, which is integrable in $[1,\infty)^{3}$. When $Hd+2H\leq2$, we have
		\begin{align}\label{sec7-eq.12}
			I_1&\leq K\int_{[1,\infty)^2}\frac{x^{2H-1}+z^{2H-1}}{(x\vee z)^{Hd+4H-2}}\d x\d z\nonumber\\
			&\leq K\int_1^\infty\int_1^x z^{2H-1}x^{-2Hd-4H+2}\d z\d x\nonumber\\
			&\leq K\int_1^\infty x^{-2Hd-2H+2}\d x<\infty,
		\end{align}
		which implies \eqref{sec4-eq.1} on $\mathcal{C}_1$ for the case $i=2$. When $(x,y,z)\in\mathcal{C}_2$, the fact that 
		$$\mu_2\leq\sqrt{\lambda_2\rho_2}=y^H(x+y+z)^H\leq Ky^{2H},$$
		in addition to the condition $Hd+H>\frac32$
		\begin{align}\label{sec7-eq.11}
			I_2&:=\int_{[1,\infty)^3}\frac{\mathds{1}_{[x\vee z,\infty]}(y)\mu_2}{[(1+\lambda_2)(1+\rho_2)-\mu_2^2]^{\frac d2+1}}\d x\d y\d z\nonumber\\
			&\leq K\int_{[1,\infty)^2}\int_{x\vee z}^\infty y^{-Hd}(x\vee z)^{-Hd-2H}\d y\d x\d z\nonumber\\
			&= K\int_{[1,\infty)^2}(x\vee z)^{-2Hd-2H+1}\d x\d z<\infty.
		\end{align}
		Consequently, we have proved \eqref{sec4-eq.1} for $i=2$ by combining \eqref{sec7-eq.12} with \eqref{sec7-eq.11}.
		\\\textbf{For} {\boldmath{$i=3$}}, we decompose the integral into two parts 
		\begin{align*}
			&\int_{\mathbb{R}_+^3}\frac{\mu_3}{[(1+\lambda_3)(1+\rho_3)-\mu_3^2]^{\frac d2+1}}\d x\d y\d z\\
            &=\left(\int_0^1+\int_1^{\infty}\right)\int_{\mathbb{R}_+^2}\frac{\mu_3}{[(1+\lambda_3)(1+\rho_3)-\mu_3^2]^{\frac d2+1}}\d x\d z\d y\\
			&=:I_1+I_2.
		\end{align*}
		For $I_1$, by the fact that 
		$$\mu_3\leq\sqrt{\lambda_3\rho_3}=(xz)^H$$
		and that
		$$(1+\lambda_3)(1+\rho_3)-\mu_3^2\geq K(1+x^{2H})(1+z^{2H}),$$
		we have
		\begin{equation*}
			I_1\leq K\int_0^1\int_{\mathbb{R}_+^2}\frac{(xz)^H}{[(1+x^{2H})(1+z^{2H})]^{\frac d2+1}}\d x\d z\d y<\infty.
		\end{equation*}
		For $I_2$, given that
		\begin{align}\label{sec7-eq.1}
			\mu_3&=H(1-2H)xz\int_0^1\int_0^1(y+xu+zv)^{2H-2}\d u\d v\nonumber\\
			&\leq Kxz\int_0^1\int_0^1(y^\beta(xu+zv)^\alpha)^{2H-2}\d u\d v\nonumber\\
			&\leq Kx^{\alpha(H-1)+1}z^{\alpha(H-1)+1}y^{2\beta(H-1)},
		\end{align}
		where $\alpha+\beta = 1$, $\alpha,\beta\in(0,1)$ and we use the Young inequality in the first and second inequality, we then obtain
		\begin{equation}\label{sec4-eq.4}
			I_2\leq K\int_1^{\infty}y^{2\beta(H-1)}\d y\int_{\mathbb{R}_+^2}\frac{x^{\alpha(H-1)+1}z^{\alpha(H-1)+1}}{[(1+x^{2H})(1+z^{2H})]^{\frac d2+1}}\d x\d z.
		\end{equation}
		If $Hd+2H>2$, it's obvious that \eqref{sec4-eq.4} is finite because $H<\frac12$. Suppose that $Hd+2H\leq2$. Let $0<\delta<\frac{2Hd+2H-3}{2}$ and set
		$$\alpha = \frac{2-2H-Hd+\delta}{1-H},\qquad\beta=\frac{Hd+H-1-\delta}{1-H}.$$
		As a consequence, we have
		$$Hd+2H+\alpha(1-H)-1=1+\delta>1,\qquad2\beta(1-H)=2Hd+2H-2-2\delta>1,$$
		which implies the finiteness of \eqref{sec4-eq.4} and thus the desired result follows.
	\end{proof}
	\begin{lemma}\label{sec7.lm.7-4}
		Suppose that $H\leq\frac{3}{2(1+d)}$ and $d\geq3$. Then, for $i=1,2,3$,
		\begin{equation*}
			\int_{0\leq\alpha+\beta\leq1}\frac{\mu_i}{(\lambda_i\rho_i-\mu_i^2)^{\frac d2+1}}\d \alpha\d\beta<\infty.
		\end{equation*} 
		where $\lambda_i$, $\rho_i$ and $\mu_i$ are the parameters $\lambda$, $\rho$ and $\mu$ specific to the domain $\D_i$ with respect to the coordinates $(a=t\e^{\frac1{2H}}\alpha,b=1-\alpha-\beta,c=t\e^{\frac1{2H}}\beta)$, respectively.
	\end{lemma}
	\begin{proof}
		\textbf{For} {\boldmath{$i=1$}}, by the fact that 
		$$\mu_1\leq\sqrt{\lambda_1\rho_1}=(1-\alpha)^H(1-\beta)^H$$
		and that
		\begin{align*}
			\lambda\rho-\mu^2&\geq K\left((1-\alpha)^{2H}\alpha^{2H}+(1-\beta)^{2H}\beta^{2H}\right)\\
			&\geq K(1-\alpha)^{H}\alpha^{H}(1-\beta)^{H}\beta^{H},
		\end{align*}
		we have 
		\begin{align*}
			\int_{0\leq\alpha+\beta\leq1}\frac{\mu_1}{(\lambda_1\rho_1-\mu_1^2)^{\frac d2+1}}\d \alpha\d\beta\leq K\left(\int_0^1(1-\alpha)^{-\frac{Hd}2}\alpha^{-\frac{Hd}2-H}\d\alpha\right)^2<\infty
		\end{align*}
		due to the condition $Hd+2H<2$.
		\\\textbf{For} {\boldmath{$i=2$}}, by Lemma \ref{sec2-lm.1} Case (ii), we have
		\begin{align*}
			\int_{0\leq\alpha+\beta\leq1}\frac{\mu_2}{[(\lambda_2\rho_2-\mu_2^2)]^{\frac d2+1}}\d \alpha\d \beta&\leq\int_{0\leq\alpha+\beta\leq1}\frac{G(\alpha,1,1-\alpha-\beta)}{[(1-\alpha-\beta)^{2H}(\alpha^{2H}+\beta^{2H})]^{\frac d2+1}}\d \alpha\d \beta
		\end{align*}
		Note that 
		$$(\alpha+\beta)^{2H}\leq2^{2H}\left(\alpha^{2H}+\beta^{2H}\right).$$
		We obtain that
		\begin{align*}
			\int_{0\leq\alpha+\beta\leq1}&\frac{\mu_2}{[(\lambda_2\rho_2-\mu_2^2)]^{\frac d2+1}}\d \alpha\d \beta\leq K\int_0^1\int_0^u\frac{G(v,1,1-u)}{(1-u)^{Hd+2H}u^{Hd+2H}}\d v\d u\\
			&=K\left(\int_0^\delta+\int_\delta^1\right)\int_0^u\frac{G(v,1,1-u)}{(1-u)^{Hd+2H}u^{Hd+2H}}\d v\d u:=I_1+I_2,
		\end{align*}
		where $0<\delta<1$ and we change the coordinates $(u=\alpha+\beta,v=\alpha)$ in the first inequality. For $I_1$, \eqref{sec6-eq.3} implies that 
		$$G(v,1,u-v)\leq K(1-u)^{2H},$$
		Then, we obtain that
		\begin{equation*}
            I_1\leq K\int_0^\delta\int_0^u(1-u)^{-Hd}u^{-Hd-2H}\d v\d u\leq K_\delta\int_0^\delta u^{-Hd-2H+1}\d u<\infty
        \end{equation*}
		because $Hd+2H-1<1$. For $I_2$, from the relation given by \eqref{sec6-eq.2}
		\begin{align*}
			G(x,x+y+z,y)\leq Ky\left(x^{2H-1}+z^{2H-1}\right),
		\end{align*}
		we deduce that
	\begin{equation*}
        I_2\leq K\int_\delta^1\int_0^u\frac{v^{2H-1}+(u-v)^{2H-1}}{(1-u)^{Hd+2H-1}u^{Hd+2H}}\d v\d u\leq K_\delta\int_\delta^1(1-u)^{-Hd-2H+1}\d u<\infty.
        \end{equation*}
        
		\textbf{For} {\boldmath{$i=3$}}, by Case (iii) in Lemma \ref{sec2-lm.1}, we are supposed to show that
		$$\int_{0\leq\alpha+\beta\leq1}\frac{G(1-\beta,\alpha,\beta)}{(\alpha\beta)^{Hd+2H}}\d \alpha\d\beta<\infty.$$
		To achieve this, we decompose the integral into two parts
		\begin{align*}
			\int_{0\leq\alpha+\beta\leq1}\frac{G(1-\beta,\alpha,\beta)}{(\alpha\beta)^{Hd+2H}}\d \alpha\d\beta&=\left(\int_0^{1-\delta}+\int_{1-\delta}^1\right)\int_0^{1-\alpha}\frac{G(1-\beta,\alpha,\beta)}{(\alpha\beta)^{Hd+2H}}\d \beta\d\alpha\\
			&=:I_1+I_2
		\end{align*}
		for a fixed and sufficiently small $\delta>0$. For $I_1$, we further break down the integral into two parts
		\begin{align*}
			I_1=\int_0^{1-\delta}\left(\int_0^{\frac{\delta}2}+\int_{\frac{\delta}2}^{1-\alpha}\right)\frac{G(1-\beta,\alpha,\beta)}{(\alpha\beta)^{Hd+2H}}\d \beta\d\alpha=:I_{11}+I_{12}.
		\end{align*}
		For $I_{11}$, we deduce from \eqref{sec7-eq.1} that 
		\begin{equation}\label{sec7-eq.2}
			I_{11}\leq K\int_0^{1-\delta}\int_0^{\frac{\delta}2}\frac{(1-\alpha-\beta)^{2H-2}}{(\alpha\beta)^{Hd+2H-1}}\d \beta\d\alpha<\infty.
		\end{equation}
		From the relation
		\begin{align}\label{sec7-eq.3}
			G(x+y,x,z)&=H(1-2H)xz\int_0^1\int_0^1(y+xu+zv)^{2H-2}\d u\d v\nonumber\\
			&\leq H(1-2H)xy^{H-1}z^H,
		\end{align}
		we obtain that
		\begin{equation}\label{sec7-eq.4}
			I_{12}\leq K\int_0^{1-\delta}\int_{\frac{\delta}2}^{1-\alpha}\frac{(1-\alpha-\beta)^{H-1}}{\alpha^{Hd+2H-1}\beta^{Hd+H}}\d \beta\d\alpha<\infty.
		\end{equation}
		For $I_2$, similar to \eqref{sec7-eq.3}, we can deduce that
		$$G(x+y,x,z)\leq H(1-2H)zy^{H-1}x^H.$$
		As a result, we obtain
		\begin{equation}\label{sec7-eq.5}
			I_2\leq K\int_{1-\delta}^1\int_0^{1-\alpha}\frac{(1-\alpha-\beta)^{H-1}}{\alpha^{Hd+H}\beta^{Hd+2H-1}}\d \beta\d\alpha<\infty.
		\end{equation}
		Combining \eqref{sec7-eq.2}, \eqref{sec7-eq.4} and \eqref{sec7-eq.5} leads to the desired result. The proof is complete.
	\end{proof}

\bigskip

\textbf{Acknowledgement} ~Q. Yu is supported by the National Natural Science Foundation of China (12201294).

\textbf{Data Availability Statements} ~The data that support the findings of this study are available from the corresponding author upon reasonable request.

\textbf{Declaration of interests} ~The authors declare that they have no known competing financial interests or personal relationships that
could have appeared to influence the work reported in this paper.

	\bibliographystyle{plain}
	\bibliography{reference.bib}
\end{document}